\documentclass{amsart}
\usepackage[utf8]{inputenc}

\usepackage{mathtools}
\usepackage{amssymb}
\usepackage{accents}
\usepackage{mathrsfs}
\usepackage{bm}
\usepackage[all]{xy}
\UseComputerModernTips
\CompileMatrices
\usepackage[svgnames]{xcolor}
\usepackage[colorlinks,allcolors=blue]{hyperref}
\usepackage{ifthen}
\usepackage{latexsym}
\usepackage{enumitem}
\setlist[enumerate,1]{label=\textup{(\arabic*)}}
\usepackage{tikz}
\usepackage{tikz-3dplot}
\usetikzlibrary{backgrounds}
\usetikzlibrary{arrows.meta}
\allowdisplaybreaks
\numberwithin{equation}{section}
\newtheorem{theorem}[equation]{Theorem}
\newtheorem{lemma}[equation]{Lemma}
\newtheorem{proposition}[equation]{Proposition}
\newtheorem{corollary}[equation]{Corollary}

\theoremstyle{definition}
\newtheorem{definition}[equation]{Definition}

\theoremstyle{remark}
\newtheorem{remark}[equation]{Remark}

\renewcommand{\phi}{\varphi}
\DeclareMathSymbol{\boxprod}{\mathbin}{AMSa}{"03} 
\DeclareMathSymbol{\mixprod}{\mathbin}{AMSa}{"4F} 

\newcommand{\corrs}{\leftrightarrow}

\newcommand{\dirsum}{\oplus}

\newcommand{\disjunion}{\sqcup}

\newcommand{\dual}{^\vee}

\newcommand{\homeo}{\approx}

\newcommand{\includesin}{\hookrightarrow}

\newcommand{\iso}{\cong}
\newcommand{\Mackey}[1]{{\underline {#1}}}

\newcommand{\Susp}{\Sigma}
\newcommand{\susp}{\Sigma}
\newcommand{\tensor}{\otimes}

\newcommand{\C}{{\mathbb C}}

\newcommand{\PP}{\mathbb{P}}

\newcommand{\Z}{\mathbb{Z}}

\newcommand{\HS}{\Mackey{\mathbb{H}}}
\newcommand{\HH}{\Mackey H_\GG}

\newcommand{\cwd}[1][]{\widehat{c}_\omega^{\ifthenelse{\equal{#1}{}}{}{{\:#1}}}}
\newcommand{\cxwd}[1][]{\widehat{c}_{\chiw}^{\ifthenelse{\equal{#1}{}}{}{{\:#1}}}}

\newcommand{\cld}[1][]{\widehat{c}_{\lambda}^{\ifthenelse{\equal{#1}{}}{}{{\:#1}}}}
\newcommand{\cxld}[1][]{\widehat{c}_{\chi\lambda}^{\ifthenelse{\equal{#1}{}}{}{{\:#1}}}}

\newcommand{\clod}[1][]{\widehat{c}_{\omega_1}^{\ifthenelse{\equal{#1}{}}{}{{\:#1}}}}
\newcommand{\cxlod}[1][]{\widehat{c}_{\chi\omega_1}^{\ifthenelse{\equal{#1}{}}{}{{\:#1}}}}
\newcommand{\cltd}[1][]{\widehat{c}_{\omega_2}^{\ifthenelse{\equal{#1}{}}{}{{\:#1}}}}
\newcommand{\cxltd}[1][]{\widehat{c}_{\chi\omega_2}^{\ifthenelse{\equal{#1}{}}{}{{\:#1}}}}
\newcommand{\cltensd}[1][]{\widehat{c}_{\omega_1\tensor\omega_2}^{\ifthenelse{\equal{#1}{}}{}{{\:#1}}}}
\newcommand{\cxltensd}[1][]{\widehat{c}_{\chi\omega_1\tensor\omega_2}^{\ifthenelse{\equal{#1}{}}{}{{\:#1}}}}
\newcommand{\cd}[1][]{\widehat{c}^{\ifthenelse{\equal{#1}{}}{}{{\:#1}}}}
\newcommand{\cgd}[1][]{\widehat{c}_\pi^{\ifthenelse{\equal{#1}{}}{}{{\:#1}}}}
\newcommand{\Cpq}[2]{\C^{#1+#2\sigma}}
\newcommand{\Cp}[1]{\C^{#1}}
\newcommand{\Cq}[1]{\C^{#1\sigma}}
\newcommand{\Xpq}[2]{\PP(\Cpq{#1}{#2})}

\newcommand{\Xp}[1]{\PP(\C^{#1})}
\newcommand{\Xq}[1]{\PP(\Cq{#1})}
\newcommand{\chiw}{\chi\omega}
\newcommand{\Grpq}[3]{\mathrm{Gr}_{#1}(\Cpq{#2}{#3})}
\newcommand{\Grp}[2]{\mathrm{Gr}_{#1}(\Cp{#2})}
\newcommand{\Grq}[2]{\mathrm{Gr}_{#1}(\Cq{#2})}

\newcommand{\Qpq}[2]{Q_{#1,#2}}

\newcommand{\Qexpq}[2]{Q(\Cpq {#1}{#2})}   
\newcommand{\Qexp}[1]{Q(\Cp {#1})}
\newcommand{\Qexq}[1]{Q(\Cq {#1})}

\renewcommand{\vec}[1]{\accentset{\rightharpoonup}{#1}} 
\newcommand{\gr}{\Diamond}      
\newcommand{\ext}{\mathsf{\Lambda}}     
\newcommand{\divq}[1][]{\phi^\chi_{#1}}
\newcommand{\rels}[1]{\langle #1 \rangle}
\newcommand{\gens}[1]{\langle #1 \rangle}

\DeclareMathOperator{\grad}{grad}

\DeclareMathOperator{\Sym}{Sym}
\DeclareMathOperator{\Gl}{Gl}
\newcommand{\GG}{{C_2}}
\begin{document}

\title[Complex quadrics III]
{The $\GG$-equivariant ordinary cohomology of complex quadrics III: Exceptional cases}
\date{\today}

\author{Steven R. Costenoble}
\address{Steven R. Costenoble\\Department of Mathematics\\Hofstra University\\
  Hempstead, NY 11549, USA}
\email{Steven.R.Costenoble@Hofstra.edu}
\author{Thomas Hudson}
\address{Thomas Hudson, College of Transdisciplinary Studies, DGIST, 
Daegu, 42988, Republic of Korea}
\email{hudson@dgist.ac.kr}



\abstract
In this, the last of three papers about $\GG$-equivariant complex quadrics,
we complete the calculation of the equivariant ordinary cohomology of smooth symmetric quadrics
in the cases where the fixed sets have more than two components.
These calculations imply one for a $\GG$-equivariant Grassmannian,
which we use to prove an equivariant refinement of the result that there are 27 lines
on a cubic surface in $\PP^3$.
\endabstract

\maketitle
\tableofcontents

\section{Introduction}

(Note: This is a preliminary version and will be rewritten so that its nomenclature and notation
is consistent with that of \cite{CH:QuadricsI}.)

This is the last of three papers about the $\GG$-equivariant ordinary cohomology
of smooth equivariant projective quadric varieties over $\C$.
The first, \cite{CH:QuadricsI}, looked at antisymmetric quadrics
and the second, \cite{CH:QuadricsII}, looked at symmetric quadrics,
but graded cohomology on $RO(\Pi BU(1))$ in all cases, glossing over the
fact that some cases can be given a larger grading because their
fixed sets have more than two components.
In this paper we look at those exceptional cases and complete the calculations
of their cohomologies with their natural, larger, gradings.

We will assume most of the background and notation from the earlier papers,
particularly the appendix in \cite{CH:QuadricsI} that summarizes what is needed
about equivariant ordinary cohomology and the cohomology of projective spaces.
We repeat here the definitions of the various types of symmetric quadrics.

\begin{definition}\ 
\begin{itemize}
    \item \emph{Quadrics of type (D,D)} are subvarieties of $\Xpq{2p}{2q}$ of the form
    \begin{multline*}
        \Qpq{2p}{2q} = \Qexpq{2p}{2q} = \\
        \bigl\{ [x_1:\cdots:x_p:u_p:\cdots:u_1:y_1:\cdots:y_q:v_q:\cdots:v_1] \in \Xpq{2p}{2q} \mid \\
                            \textstyle \sum_i x_iu_i + \sum_j y_jv_j = 0 \bigr\}.
    \end{multline*}

    \item \emph{Quadrics of type (B,D)} are subvarieties of $\Xpq{(2p+1)}{2q}$ of the form
    \begin{multline*}
        \Qpq{2p+1}{2q} = \Qexpq{(2p+1)}{2q} = \\
        \bigl\{ [x_1:\cdots:x_p:z:u_p:\cdots:u_1:y_1:\cdots:y_q:v_q:\cdots:v_1] \in \Xpq{(2p+1)}{2q} \mid \\
                            \textstyle \sum_i x_iu_i + z^2 + \sum_j y_jv_j = 0 \bigr\}.
    \end{multline*}

    \item \emph{Quadrics of type (D,B)} are subvarieties of $\Xpq{2p}{(2q+1)}$ of the form
    \begin{multline*}
        \Qpq{2p}{2q+1} = \Qexpq{2p}{(2q+1)} = \\
        \bigl\{ [x_1:\cdots:x_p:u_p:\cdots:u_1:y_1:\cdots:y_q:w:v_q:\cdots:v_1] \in \Xpq{2p}{(2q+1)} \mid \\
                            \textstyle \sum_i x_iu_i + \sum_j y_jv_j + w^2 = 0 \bigr\}.
    \end{multline*}
\end{itemize}
\end{definition}

As in the earlier papers, we will often use the abbreviated notation
\[
    [\vec x:\vec u:\vec y:w:\vec v] = [x_1:\cdots:x_p:u_p:\cdots:u_1:y_1:\cdots:y_q:w:v_q:\cdots:v_1]
\]
and the obvious variations for the other types.

We calculated the cohomologies of these quadrics (and a fourth type, (B,B)) in \cite{CH:QuadricsII},
but with grading $RO(\Pi BU(1))$ in all cases.
That is appropriate for $\Qpq mn$ when $m, n > 2$, as $\Qpq mn^\GG$ then has two components,
so $RO(\Pi \Qpq mn) \iso RO(\Pi BU(1))$ via the inclusions
$\Qpq mn\includesin \Xpq mn\includesin B U(1)$.
However, if $m=2$ or $n=2$, the fixed set has more than two components and
the induced map $RO(\Pi BU(1))\to RO(\Pi\Qpq mn)$ is no longer surjective.
This paper finishes the calculations in those cases by computing the cohomology with
the full $RO(\Pi\Qpq mn)$ grading.

Because $\Qpq mn^\GG$ will have more than two components, we need to label them carefully,
and our choice of labeling is explained in \S\ref{sec:notations}.
We complete the calculation of the cohomology of $\Qpq 2{2q+1}$ in Theorem~\ref{thm:db multiplicative} and
the calculation of the cohomology of $\Qpq 2{2q}$ for $q > 1$ in Theorem~\ref{thm:dd multiplicative}.
The cohomology of $\Qpq 22$ is determined in Theorem~\ref{thm:22 multiplicative}.

One interesting case is $\Qpq 24$, which is $\GG$-diffeomorphic to $\Grpq 222$,
the Grassmannian of 2-planes in $\Cpq 22$,
so our calculations also give the cohomology
of this Grassmannian. We use this in \S\ref{sec:27 lines} to give an equivariant generalization
of the classical result that there are exactly 27 lines on a smooth cubic hypersurface
in $\PP^3$, looking at cubic surfaces in $\Xpq 22$.
This duplicates and refines a calculation done by
Brazelton in \cite{Braz:equivenumerative}.
We obtained a similar result in \cite{CH:QuadricsI}, for cubic surfaces in $\Xpq 3{}$,
and gave a detailed comparison of our approach and Brazelton's there.

\section{Notations and conventions}\label{sec:notations}

Our notation for the components of the fixed sets of $\Qpq mn$ and the resulting elements
in its representation ring are based on the case with the most components, $\Qpq 22$.
We use the identification of $\Qpq 22$ with $\Xpq 1{}\times\Xpq 1{}$
given by the Segre embedding. To be precise, we modify the Segre embedding slightly by introducing a sign, defining
\begin{gather*}
    s\colon \Xpq 1{}\times \Xpq 1{} \to \Xpq 22 \\
    s([a_1:a_2],[b_1:b_2]) = [a_1 b_1: a_2b_2: a_1b_2: -a_2b_1],
\end{gather*}
whose image is precisely $\Qpq 22$.
Writing $P = \Xpq 1{}\times \Xpq 1{}$ for brevity, $P^\GG$ consists of four points
that are naturally labeled as follows, consistent with our notation in
\cite{CH:bt2}:
\[
    P^\GG = P^{00}\disjunion P^{01}\disjunion P^{10}\disjunion P^{11} 
\]
where
\begin{alignat*}{2}
    P^{00} &= \Xp{}\times\Xp{} &&= ([1:0],[1:0]) \\
    P^{01} &= \Xp{}\times\Xq{} &&= ([1:0],[0:1]) \\
    P^{10} &= \Xq{}\times\Xp{} &&= ([0:1],[1:0]) \\
    P^{11} &= \Xq{}\times\Xq{} &&= ([0:1],[0:1]).
\end{alignat*}
We label the components of $\Qpq 22^\GG$ to match:
\[
    \Qexpq 22^\GG = \Qpq 22^{00} \disjunion \Qpq 22^{01} \disjunion \Qpq 22^{10} \disjunion \Qpq 22^{11}
\]
where
\begin{align*}
    \Qpq 22^{00} &= [1:0:0:0] \\
    \Qpq 22^{01} &= [0:0:1:0] \\
    \Qpq 22^{10} &= [0:0:0:1] \\
    \Qpq 22^{11} &= [0:1:0:0].
\end{align*}
Thus,
\[
    RO(\Pi\Qpq 22) = \Z\{1,\sigma,\Omega_{00},\Omega_{01},\Omega_{10},\Omega_{11}\}/
        \rels{\textstyle \sum_{i,j} \Omega_{ij} = 2\sigma - 2},
\]
and the map $RO(\Pi B U(1)) \to RO(\Pi\Qpq 22)$
induced by $\Qpq 22\includesin \Xpq 22\includesin BU(1)$
is given by
\[
    \Omega_0 \mapsto \Omega_{00} + \Omega_{11} \qquad\text{and}\qquad
    \Omega_1 \mapsto \Omega_{01} + \Omega_{10}.
\]

We label the fixed set components of the other cases to be consistent with this one.
If $n > 2$ we write
\[
    \Qexpq 2n^\GG = \Qpq 2n^{00} \disjunion \Qpq 2n^{11} \disjunion \Qpq 2n^1
\]
where
\begin{align*}
    \Qpq 2n^{00} &= [1:0:\vec 0] \\
    \Qpq 2n^{11} &= [0:1:\vec 0] \\
    \Qpq 2n^1 &= \Qexq n.
\end{align*}
The representation ring is then
\[
    RO(\Qpq 2n) = \Z\{1,\sigma,\Omega_{00},\Omega_{11},\Omega_1\}/
        \rels{\Omega_{00} + \Omega_{11} + \Omega_1 = 2\sigma - 2}
\]
with the map $RO(\Pi B U(1)) \to RO(\Pi\Qpq 2n)$ given by
\[
    \Omega_0 \mapsto \Omega_{00} + \Omega_{11} \qquad\text{\and}\qquad \Omega_1 \mapsto \Omega_1.
\]

Similarly, if $m > 2$,
\[
    \Qexpq m2^\GG = \Qpq m2^0 \disjunion \Qpq m2^{01} \disjunion \Qpq m2^{10}
\]
where
\begin{align*}
    \Qpq m2^0 &= \Qexp m \\
    \Qpq m2^{01} &= [\vec 0:1:0] \\
    \Qpq m2^{10} &= [\vec 0:0:1]
\end{align*}
and we use the implied notation for the elements of the representation ring.

We will also be using inclusions of projective spaces corresponding to certain maximal isotropic affine
subspaces of $\Cpq mn$. For $\Qpq 22$, a relatively natural labeling of these maps is as follows.
\begin{equation}\label{eqn:inclusions}
    \begin{alignedat}{3}
    i_0&\colon \Xpq 1{}\includesin \Qpq 22 &\qquad&& i_0[x:y] &= [x:0:y:0] \\
    i_1&\colon \Xpq 1{}\includesin \Qpq 22 &&& i_1[x:y] &= [x:0:0:y] \\
    i_2&\colon \Xpq 1{}\includesin \Qpq 22 &&& i_2[x:y] &= [0:x:y:0] \\
    i_3&\colon \Xpq 1{}\includesin \Qpq 22 &&& i_3[x:y] &= [0:x:0:y]    
    \end{alignedat}
\end{equation}
We extend this to the other cases in a straightforward way as we consider them.

\section{Quadrics of type (B,D) and (D,B)}

Because of the diffeomorphism $\Qpq{2}{2q+1} \homeo \Qpq{2q+1}{2}$,
it suffices to calculate the cohomology of $\Qpq{2}{2q+1}$ here.
Recall the decomposition
\[
    \Qpq{2}{2q+1}^\GG = \Qpq{2}{2q+1}^{00} \disjunion \Qpq{2}{2q+1}^{11}\disjunion \Qpq{2}{2q+1}^1
\]
introduced in \S\ref{sec:notations}.

\begin{remark}
We know that we have elements
\begin{align*}
    \zeta_{00} &\in \HH^{\Omega_{00}}(\Qpq 2{2q+1}) \\
    \zeta_{11} &\in \HH^{\Omega_{11}}(\Qpq 2{2q+1}) \\
    \zeta_{1} &\in \HH^{\Omega_{1}}(\Qpq 2{2q+1})
\end{align*}
corresponding to the three components of $\Qpq 2{2q+1}^\GG$. We will
often write $\zeta_0 = \zeta_{00}\zeta_{11}$ as shorthand.
It is the pullback of $\zeta_0$ from $\Xpq 2{(2q+1)}$.
\end{remark}

\begin{remark}
Because $\Qpq 2{2q+1}$ has three components (when $q = 0$, we take the third component as there, but empty),
the fixed-point map lands in a direct sum of three groups:
\[
    (-)^\GG\colon \HH^\alpha(\Qpq 2{2q+1}) \to 
    \HH^{\alpha_{00}^\GG}(\Qpq 2{2q+1}^{00})\dirsum\HH^{\alpha_{11}^\GG}(\Qpq 2{2q+1}^{11})
    \dirsum\HH^{\alpha_1^\GG}(\Qpq 2{2q+1}^1),
\]
where $\alpha = (\alpha_{00},\alpha_{11},\alpha_1)$ gives the restrictions to $RO(\GG)$
of $\alpha$ on each component, and $\alpha_{00}^\GG$ is the dimension of the fixed set representation, and so on.
As a result, if $z$ is a cohomology element, we will write $z^\GG$
as a triple, always in the order given above.
\end{remark}

\begin{remark}\label{rem:x11 case}
We include the case $q = 0$, where we can identify $\Qpq 21 \homeo \Xpq 1{}$ via the
map $\Xpq 1{}\to \Qpq 21$ given by
\[
    [a:b] \mapsto [a^2:b^2:iab].
\]
We can identify the composite $f\colon\Xpq 1{}\to \Qpq 21\includesin \Xpq 2{}$
up to some irrelevant changes with the composite
of the diagonal map and the Segre embedding that classifies tensor products of line bundles.
This means that 
$f^*\omega$ is not the tautological bundle $\lambda$ over $\Xpq 1{}$ but, instead,
$f^*\omega = \lambda^{\tensor 2}$, or $f^* O(1) = O(2)$,
In Remark~\ref{rem:x11 cohomology} we will compare our calculation here to the
cohomology of $\Xpq 1{}$ as computed in \cite{CHTFiniteProjSpace}.
\end{remark}

In \cite{CH:QuadricsII} we used two inclusions of $\Xpq pq$ in $\Qpq{2p}{2q+1}$.
Here, we need to consider four. Define the following maps, consistent with those in (\ref{eqn:inclusions}).
\begin{alignat*}{3}
    i_0\colon \Xpq 1q &\includesin \Qpq{2}{2q+1} &\qquad&& i[x:\vec y] &= [x:0:\vec y:0:\vec 0] \\
    i_1\colon \Xpq 1q &\includesin \Qpq{2}{2q+1} &&& i[x:\vec y] &= [x:0:\vec 0:0:\vec y] \\
    i_2\colon \Xpq 1q &\includesin \Qpq{2}{2q+1} &&& i[x:\vec y] &= [0:x:\vec y:0:\vec 0] \\
    i_3\colon \Xpq 1q &\includesin \Qpq{2}{2q+1} &&& i[x:\vec y] &= [0:x:\vec 0:0:\vec y].
\end{alignat*}
The normal bundle to each of these maps is isomorphic to
$\nu = \omega\dual \dirsum (q+1)\chi\omega\dual - O(2)$,
which has dimension $\omega + (q+1)\chi\omega\dual - 2$, which we also denote by $\nu$.
Elaborating slightly on a result of \cite{CH:QuadricsII}, we have the following.

\begin{proposition}\label{prop:db cofibration}
We have cofibration sequences
\begin{align*}
    \Xpq 1q_+ &{}\xrightarrow{i_0} (\Qpq{2}{2q+1})_+ \to \susp^\nu i_3(\Xpq 1q)_+ \\
\intertext{and}
    \Xpq 1q_+ &{}\xrightarrow{i_2} (\Qpq{2}{2q+1})_+ \to \susp^\nu i_1(\Xpq 1q)_+
\end{align*}
\end{proposition}

\begin{proof}
The projection
\begin{align*}
   \Qpq{2}{2q+1} \setminus i_0(\Xpq 1q) &\to i_3(\Xpq 1q), \\
   [ x: u:\vec y:w:\vec v] &\mapsto [ 0: u:\vec 0:0:\vec v]
\end{align*}
can be identified with the normal bundle $\nu$, and similarly for $i_2$ and $i_1$.
\end{proof}

This implies two long exact sequences in cohomology. In the $RO(\Pi B U(1))$ grading,
we saw in \cite{CH:QuadricsII} that the one using $i_0$ and $i_3$
splits into short exact sequences, and the same is true for the other, in this grading.
In the larger grading $RO(\Pi\Qpq 2{2q+1})$, this is no longer true.
Instead, we shall see that each of the long exact sequences gives a short exact sequence in only certain gradings.

Let
\begin{align*}
    x_{1,q} &= (i_3)_!(1) = [i_3(\Xpq 1q)]^* \in \HH^\nu(\Qpq 2{2q+1}) \\
\intertext{and}
    x'_{1,q} &= (i_1)_!(1) = [i_1(\Xpq 1q)]^* \in \HH^\nu(\Qpq 2{2q+1}).
\end{align*}
The element $x_{1,q}$ is the one we called $x_{p,q}$ in general in \cite{CH:QuadricsII}.
If $p > 1$, then the analogue of $x'_{1,q}$, which we might call $x'_{p,q}$, would be the same element as 
$x_{p,q}$, but when $p=1$ they are different, as we can see by looking at
their fixed points.
Think of $x_{1,q}'$ as a temporary name, as we shall shortly see how to write it in
terms of $x_{1,q}$.
We note here the following.

\begin{lemma}\label{lem:db x divisibility}
$x_{1,q}$ is infinitely divisible by $\zeta_{00}$ and $x'_{1,q}$ is infinitely divisible by $\zeta_{11}$.
\end{lemma}

\begin{proof}
$i_3^*(\zeta_{11})$ is invertible in the cohomology of $\Xpq 1q$ because
$i_3^{-1}(\Qpq 1q^{00}) = \emptyset$.
Hence
\[
    \zeta_{00}^k \cdot (i_3)_!(\zeta_{00}^{-k}) = (i_3)_!(1) = x_{1,q},
\]
so $x_{1,q}$ is infinitely divisible by $\zeta_{00}$.
The argument for $x'_{1,q}$ is similar.
\end{proof}

In \cite{CH:QuadricsII} we defined an element which here would be written
\[
    \divq = \cxwd[q] - e^{-2}\kappa x_{1,q}.
\]
In \cite[Proposition~6.2]{CH:QuadricsII} 
we proved that
\begin{alignat*}{2}
    \cwd &\text{ is infinitely divisible by $\zeta_0$ and} \\
    \divq &\text{ is infinitely divisible by $\zeta_1$.}
\end{alignat*}
(Note that, if $q = 0$, then $\zeta_1$ is invertible because
the corresponding component of the fixed set is $\Qpq 21^1 = \emptyset$.)

We are now ready to see that we have the short exact sequences mentioned above.
Note that the cosets of $RO(\Pi BU(1))$ in $RO(\Pi\Qpq 2{2q+1})$
are the sets $m\Omega_{11} + RO(\Pi BU(1))$ for $m\in\Z$.

\begin{proposition}\label{prop:db splitting}
If $m \geq 0$, then there is a split short exact sequence
\begin{multline*}
    0 \to \susp^{\nu-m\Omega_0}\HH^{RO(\Pi BU(1))}(\Xpq 1q) \xrightarrow{(i_3)_!}
    \HH^{m\Omega_{11} + RO(\Pi BU(1))}(\Qpq 2{2q+1}) \\ \xrightarrow{i_0^*}
    \HH^{RO(\Pi BU(1))}(\Xpq 1q) \to 0.
\end{multline*}
If $m \leq 0$, then there is a split short exact sequence
\begin{multline*}
    0 \to \susp^\nu\HH^{RO(\Pi BU(1))}(\Xpq 1q) \xrightarrow{(i_1)_!}
    \HH^{m\Omega_{11} + RO(\Pi BU(1))}(\Qpq 2{2q+1}) \\ \xrightarrow{i_2^*}
    \susp^{-m\Omega_0}\HH^{RO(\Pi BU(1))}(\Xpq 1q) \to 0.
\end{multline*}
\end{proposition}

\begin{proof}
We need to show that $i_0^*$ and $i_2^*$ are surjective in the gradings indicated.
The splitting then follows from the freeness of the targets.
Surjectivity is clear if $q = 0$, because then $\Xpq 1q = \Xp{}$ is just a point. So assume that $q > 0$.

The effect $i_0^*$ has on gradings and cohomology is given by
\begin{align*}
    i_0^*(\Omega_{00}) &= \Omega_0 & i_0^*(\Omega_{11}) &= 0 & i_0^*(\Omega_1) &= \Omega_1 \\
    i_0^*(\zeta_{00}) &= \zeta_0 & i_0^*(\zeta_{11}) &= 1 & i_0^*(\zeta_1) &= \zeta_1.
\end{align*}
The effect on gradings shows that the target of $i_0^*$ is as given in the statement of the proposition.
When $m = 0$,
since $i_0^*(\cwd) = \cwd$ and $i_0^*(\divq) = \cxwd[q]$, the divisibility of $\cwd$ by $\zeta_0$ and $\divq$ by $\zeta_1$
shows that we can find expressions in $\zeta_0$, $\zeta_1$,
$\cwd$, $\cxwd$, and $\divq$ in gradings $RO(\Pi BU(1))$
mapping to any basis element of the cohomology of $\Xpq 1q$.
When $m > 0$, we can then multiply by
$\zeta_{11}^m$ to get elements in grading $m\Omega_{11} + RO(\Pi BU(1))$
mapping to basis elements under $i_0^*$, because $i_0^*(\zeta_{11}) = 1$.

Turning to $i_2^*$ and assuming $m\leq 0$, we have
\begin{align*}
    i_2^*(\Omega_{00}) &= 0 & i_2^*(\Omega_{11}) &= \Omega_0 & i_2^*(\Omega_1) &= \Omega_1 \\
    i_2^*(\zeta_{00}) &= 1 & i_2^*(\zeta_{11}) &= \zeta_0 & i_2^*(\zeta_1) &= \zeta_1.
\end{align*}
This shows that the target of $i_2^*$ is as claimed.
Again, when $m=0$, we can find elements in gradings $RO(\Pi BU(1))$ mapping to a basis under $i_2^*$.
When $m < 0$, we can then multiply those elements by $\zeta_{00}^{-m}$ to get elements
in gradings $m\Omega_{11} + RO(\Pi BU(1))$ (because $\Omega_{00} + \Omega_{11}\in RO(\Pi BU(1))$)
mapping to a basis under $i_2^*$.
\end{proof}

This gives the additive structure of the cohomology, which we can write as follows,
though we need to state carefully what the grading means.

\begin{corollary}\label{cor:db additive}
Additively,
\[
    \HH^\gr(\Qpq{2}{2q+1}) \iso \HH^\gr(\Xpq 1q) \dirsum \susp^\nu\HH^\gr(\Xpq 1q)
\]
where $\nu = \omega + (q+1)\chi\omega - 2$.
In gradings $m\Omega_{11} + RO(\Pi BU(1))$ with $m\geq 0$, 
the first summand is regraded via $i_0^*$ while the second is regraded via $i_3^*$;
if $m < 0$, the first summand is regraded via $i_2^*$ while the second is regraded via $i_1^*$.
\qed
\end{corollary}


This implies bases over $\HS$, but we need to be a bit careful about the names we
use for the basis elements.
First, when writing down basis elements coming from the first summand via $i_0^*$ or $i_2^*$,
we should use $\divq$ in place of $\cwd[q]$.
For example, in gradings $-(q+1)\Omega_1 + RO(\GG)$, these basis elements would be
\[
    \{ \zeta_0^{q+1},\ \zeta_0^q\cxwd,\ \ldots,\ \zeta_0^2\cxwd[q-1],\ \zeta_1^{-1}\divq \}.
\]
Second, we can write the basis elements corresponding to the second summand,
$\susp^\nu\HH^\gr(\Xpq 1q)$, as multiples of $(i_3)_!(1) = x_{1,q}$ or $(i_1)_!(1) = x'_{1,q}$
as appropriate, but to put things in the correct grading we need to take advantage of the fact
that $x_{1,q}$ is divisible by $\zeta_{00}$ and $x'_{1,q}$ is divisible by $\zeta_{11}$,
per Lemma~\ref{lem:db x divisibility}.
For example, in gradings $\Omega_{11} - (q-1)\Omega_0 + RO(\GG)$,
these basis elements would be
\begin{multline*}
    \{ \zeta_{00}^{-1}x_{1,q},\ \zeta_{11}\cwd x_{1,q},\ \zeta_{00}^{-1}\cwd\cxwd x_{1,q}, \\
    \zeta_{00}^{-1}\zeta_{0}^{-1}\cwd\cxwd[2] x_{1,q},\ \ldots,\ \zeta_{00}^{-1}\zeta_0^{-(q-2)}\cwd\cxwd[q-1]x_{1,q} \}.
\end{multline*}
Of course, since $\zeta_0 = \zeta_{00}\zeta_{11}$, there are several other ways of writing many of these elements.

We said earlier that $x_{1,q}'$ was meant as a temporary name.
Since $x_{1,q}$ and $x'_{1,q}$ both live in grading $\nu \in RO(\Pi BU(1))$,
we must be able to write $x'_{1,q}$ in terms of $x_{1,q}$.
In fact, we have the following.

\begin{lemma}\label{lem:db xs}
$x'_{1,q} = x_{1,q} + e^2\divq = (1-\kappa)x_{1,q} + e^2\cxwd[q]$.
\end{lemma}

\begin{proof}
We first note that the last two expressions given are equal by the definition of $\divq$.
We need to show that they are equal to $x_{1,q}'$.

The elements $x_{1,q}$ and $x_{1,q}'$ both live in grading $\nu\in -q\Omega_1 + RO(\GG)$.
The basis given by Corollary~\ref{cor:db additive} in that coset of gradings is
\begin{multline*}
    \{ \zeta_0^q,\ \zeta_0^{q-1}\cxwd,\ \ldots,\ \zeta_0 \cxwd[q-1],\ \divq, \\
    x_{1,q},\ \zeta_0\cwd x_{1,q},\ \cwd\cxwd x_{1,q},\ \zeta_0^{-1}\cwd\cxwd[2] x_{1,q},\ \ldots,\ 
    \zeta_0^{-(q-2)}\cwd\cxwd[q-1] x_{1,q} \}.
\end{multline*}
If $q > 0$, there are three basis elements that could contribute to the group in
grading $\nu$, namely $\divq$, $x_{1,q}$, and $\zeta_0\cwd x_{1,q}$.
Precisely,
\[
    x'_{1,q} = \alpha x_{1,q} + \beta e^2\divq + \gamma e^{-2}\kappa \zeta_0\cwd x_{1,q}
\]
for some $\alpha\in A(\GG)$ and $\beta,\gamma\in\Z$.
To determine these coefficients, we first apply $\rho$:
\begin{align*}
    \rho(x'_{1,q}) &= y \\
    \rho(x_{1,q}) &= y \\
    \rho(e^2\divq) &= 0 \\
    \rho(e^{-2}\kappa\zeta_0\cwd x_{1,q}) &= 0.
\intertext{Thus, $\rho(\alpha) = 1$. Now we take fixed points:}
    (x'_{1,q})^\GG &= (1,0,y) \\
    x_{1,q}^\GG &= (0,1,y) \\
    (e^2\divq)^\GG &= (1,1,c^q) - 2(0,1,y) \\ &= (1,-1,0) \\
    (e^{-2}\kappa \zeta_0\cwd x_{1,q})^\GG &= (0,0,y).
\end{align*}
From this we see that $\alpha^\GG = 1$, $\beta = 1$, and $\gamma = 0$.
This implies that $\alpha = 1$, which gives the result of the lemma.

If $q = 0$, the basis is $\{1, x_{1,0}\}$ and we have
\[
    x'_{1,0} = \alpha x_{1,0} + \beta e^2
\]
for some $\alpha\in A(\GG)$ and $\beta\in\Z$. Applying $\rho$ and $(-)^\GG$, we get
\begin{align*}
    \rho(x'_{1,0}) &= y \\
    \rho(x_{1,0}) &= y \\
    \rho(e^2) &= 0 \\
    (x'_{1,0})^\GG &= (1,0,0) \\
    x_{1,0}^\GG &= (0,1,0) \\
    (e^2)^\GG &= (1,1,0)
\end{align*}
From this we get $\rho(\alpha) = 1$, $\alpha^\GG = -1$, and $\beta = 1$, hence
$\alpha = 1-\kappa$ and
\[
    x'_{1,0} = (1-\kappa)x_{1,0} + e^2,
\]
again verifying the lemma.
\end{proof}

If we take the equality in the form
\[
    x'_{1,q} = (1-\kappa)x_{1,q} + e^2\cxwd[q],
\]
multiply by $1-\kappa$ and rearrange, we get the symmetric
\[
    x_{1,q} = (1-\kappa)x'_{1,q} + e^2\cxwd[q].
\]
But the main takeaway from this lemma is that we do not need both $x_{1,q}$ and $x'_{1,q}$ to generate
the cohomology of $\Qpq 2{2q+1}$, all basis elements could be rewritten in terms of just $x_{1,q}$.

We complete the calculation by giving the multiplicative structure.
We need the following observation.

\begin{lemma}\label{lem:annihilator}
Suppose $z\in \HH^{RO(\Pi BU(1))}(\Qpq 2{2q+1})$
Then $x'_{1,q}z = 0$ if and only if $z\in \gens{x_{1,q}}$.
\end{lemma}

\begin{proof}
\[
    x'_{1,q}z = (i_1)_!(i_1^*(z)) = (i_1)_!(i_0^*(z))
\]
because $i_1$ and $i_0$ are $\GG$-homotopic. Since we know from Proposition~\ref{prop:db splitting}
that $(i_1)_!$ is injective in this grading, we get that $x'_{1,q}z = 0$ if and only if $i_0^*(z) = 0$.
From that same proposition we have that the kernel of $i_0^*$ equals the image of $(i_3)_!$,
which is the ideal
generated by $x_{1,q}$.
\end{proof}

In particular, $x_{1,q}x'_{1,q} = 0$, which we could also show directly from Lemma~\ref{lem:db xs}
and the known structure of $\HH^{RO(\Pi BU(1))}(\Qpq 2{2q+1})$ from \cite{CH:QuadricsII}.

We take advantage of our computations in \cite{CH:QuadricsII} by first 
finding the structure of $\HH^\gr(\Qpq 2{2q+1})$ as an algebra over the
$RO(\Pi BU(1))$-graded part.

\begin{proposition}
As an algebra over $\HH^{RO(\Pi BU(1))}(\Qpq 2{2q+1})$,
the $RO(\Pi\Qpq 2{2q+1})$-graded cohomology $\HH^\gr(\Qpq 2{2q+1})$ is generated by
\begin{align*}
    \zeta_{00} &\in \HH^{\Omega_{00}}(\Qpq 2{2q+1}) \\
    \zeta_{11} &\in \HH^{\Omega_{11}}(\Qpq 2{2q+1})
\end{align*}
subject to the facts that
\begin{alignat*}{2}
    &x_{1,q} &\quad&\text{is infinitely divisible by $\zeta_{00}$ and} \\
    &x_{1,q} + e^2\divq &&\text{is infinitely divisible by $\zeta_{11}$,} 
\end{alignat*}
and the relation
\[
    \zeta_{00}\zeta_{11} = \zeta_0.
\]
\end{proposition}

\begin{proof}
We have already shown the divisibility properties and mentioned the relation $\zeta_{00}\zeta_{11} = \zeta_0$.
it remains to show that these suffice to determine the algebra structure.
Let $A$ be the $RO(\Pi\Qpq 2{2q+1})$-graded ring
\[
    A = \HH^{RO(\Pi BU(1))}(\Qpq 2{2q+1})[\zeta_{00},\zeta_{11}, \zeta_{00}^{-k}x_{1,q}, \zeta_{11}^{-k}(x_{1,q}+e^2\divq)]
    /\rels{\zeta_{00}\zeta_{11} - \zeta_0},
\]
where we implicitly impose the relations $\zeta_{00}\cdot\zeta_{00}^{-k}x_{1,q} = \zeta_{00}^{-(k-1)}x_{1,q}$
and similarly for the other quotients.
We want to show that the ring map $A\to B = \HH^\gr(\Qpq 2{2q+1})$ taking generators to elements of the same name is an isomorphism.

We work one coset of $RO(\Pi BU(1))$ in $RO(\Pi\Qpq 2{2q+1})$ at a time.
We first observe that
\[
    A^{RO(\Pi BU(1))} \iso \HH^{RO(\Pi BU(1))}(\Qpq 2{2q+1}).
\]
This follows from the fact that we've added new elements only outside of the $RO(\Pi BU(1))$ grading:
Any monomial involving the new elements adjoined that lies in the $RO(\Pi BU(1))$ grading must
have the net powers of $\zeta_{00}$ and $\zeta_{11}$ be the same, hence the term is
equal to an element of $\HH^{RO(\Pi BU(1))}(\Qpq 2{2q+1})$.
We also know that we have added no relations that are not already true in $\HH^\gr(\Qpq 2{2q+1})$.

Now consider a coset
$m\Omega_{11} + RO(\Pi BU(1))$ with $m > 0$. Define the following two
$\HH^{RO(\Pi BU(1))}(\Qpq 2{2q+1})$-submodules generated by the element called $\zeta_{00}^{-m}x_{1,q}$ in each of $A$ and $B$:
\begin{align*}
    J &= A^{RO(\Pi BU(1))}\cdot\zeta_{00}^{-m}x_{1,q} \subset A^{m\Omega_{11}+RO(\Pi BU(1))} \\
    K &= \HH^{RO(\Pi BU(1))}(\Qpq 2{2q+1})\cdot\zeta_{00}^{-m}x_{1,q} \subset B^{m\Omega_{11}+RO(\Pi BU(1))}
\end{align*}
We get the following commutative diagram of $\HH^{RO(\Pi BU(1))}(\Qpq 2{2q+1})$-modules.
\[
    \xymatrix{
        0 \ar[r] & J \ar[r] \ar[d] & A^{m\Omega_{11}+RO(\Pi BU(1))} \ar[r] \ar[d] & A^{m\Omega_{11}+RO(\Pi BU(1))}/J \ar[r] \ar[d] & 0 \\
        0 \ar[r] & K \ar[r] & B^{m\Omega_{11}+RO(\Pi BU(1))} \ar[r] & B^{m\Omega_{11}+RO(\Pi BU(1))}/K \ar[r] & 0
    }
\]
Proposition~\ref{prop:db splitting} implies that
\begin{align*}
    K &\iso \susp^{\nu-m\Omega_0} \HH^{RO(\Pi BU(1))}(\Xpq 1q) \\
\intertext{and}
    B^{m\Omega_{11}+RO(\Pi BU(1))}/K &\iso \HH^{RO(\Pi BU(1))}(\Xpq 1q).
\end{align*}
The divisibility of $x_{1,q}$ by $\zeta_{00}$ implies that $J$ is isomorphic, with a shift in grading, to the
submodule $\gens{x_{1,q}}$ of $A^{RO(\Pi BU(1))}$, which we know is isomorphic to
$\susp^{\nu} \HH^{RO(\Pi BU(1))}(\Xpq 1q)$ by the proof of
\cite[Theorem~6.7]{CH:QuadricsII}, hence
\[
    J \iso \susp^{\nu-m\Omega_0} \HH^{RO(\Pi BU(1))}(\Xpq 1q) \iso K.
\]

On the other hand, consider multiplication by $\zeta_{11}^m$ as a map
\[
    \zeta_{11}^m\cdot \colon A^{RO(\Pi BU(1))}/\gens{x_{1,q}} \to  A^{m\Omega_{11}+RO(\Pi BU(1))}/J,
\]
which is defined because $\zeta_{11}^m x_{1,q} = \zeta_0^m\zeta_{00}^{-m}x_{1,q}$.
We claim this is an isomorphism.
That it is surjective follows because any monomial in $A^{m\Omega_{11}+RO(\Pi BU(1))}$ is
a multiple of either $\zeta_{11}^m$ or $\zeta_{00}^{-m}x_{1,q}$.
That it is injective follows from Lemma~\ref{lem:annihilator} and the fact that we know
that $A^{RO(\Pi BU(1))} \iso \HH^{RO(\Pi BU(1))}(\Qpq 2{2q+1})$:
If $\zeta_{11}^m z = 0$, then
\begin{align*}
    \zeta_{11}^m z &= w \zeta_{00}^{-m}x_{1,q} \\
\intertext{for some $w\in A^{RO(\Pi BU(1))}$, so}
    z x'_{1,q} &= w \zeta_{00}^{-m} x_{1,q} \cdot \zeta_{11}^{-m} x'_{1,q} = 0 
\end{align*}
because $x_{1,q}x'_{1,q} = 0$. But then the lemma shows that $z\in \gens{x_{1,q}}$.
The proof of \cite[Theorem~6.7]{CH:QuadricsII} showed that
\[
    A^{RO(\Pi BU(1))}/\gens{x_{1,q}} \iso \HH^{RO(\Pi BU(1))}(\Xpq 1q)
\]
hence
\[
    A^{m\Omega_{11}+RO(\Pi BU(1))}/J \iso \HH^{RO(\Pi BU(1))}(\Xpq 1q) \iso B^{m\Omega_{11}+RO(\Pi BU(1))}/K.
\]
(with no shift in grading because we are implicitly comparing gradings
via $i_0^*$, and $i_0^*(\Omega_{11}) = 0$).
The commutative diagram above then implies that
$A\iso B$ in gradings $m\Omega_{11} + RO(\Pi BU(1))$.

When $m < 0$ we have a similar argument, replacing $\zeta_{00}^{-m}x_{1,q}$ with
$\zeta_{11}^m(x_{1,q} + e^2\divq)$.
Hence $A\iso B$ in all gradings.
\end{proof}

Combining this proposition with the structure of $\HH^{RO(\Pi BU(1))}(\Qpq 2{2q+1})$ shown in
\cite{CH:QuadricsII}, we get the following.

\begin{theorem}\label{thm:db multiplicative}
As an algebra over $\HH^{RO(\Pi BU(1))}(B U(1))$, 
the $RO(\Pi\Qpq 2{2q+1})$-graded cohomology $\HH^\gr(\Qpq{2}{2q+1})$
is generated by
\begin{align*}
    \zeta_{00} &\in \HH^{\Omega_{00}}(\Qpq 2{2q+1}) \\
    \zeta_{11} &\in \HH^{\Omega_{11}}(\Qpq 2{2q+1}) \\
    x_{1,q} &\in \HH^{q\chi\omega + 2\sigma}(\Qpq{2}{2q+1})
\end{align*}
subject to the facts that
\begin{alignat*}{2}
    &\cwd &\quad&\text{is infinitely divisible by $\zeta_0$,} \\
    &\divq = \cxwd[q] - e^{-2}\kappa x_{1,q} &&\text{is infinitely divisible by $\zeta_1$,} \\
    &x_{1,q} &&\text{is infinitely divisible by $\zeta_{00}$, and} \\
    &x_{1,q} + e^2\divq &&\text{is infinitely divisible by $\zeta_{11}$,} 
\end{alignat*}
and the following relations:
\begin{align*}
    \zeta_{00}\zeta_{11} &= \zeta_0 \\
    x_{1,q}^2 &= e^2\cxwd[q]x_{1,q} \\
   \cwd\divq &= \tau(\iota^{-2})\zeta_1  x_{1,q}.
\end{align*}
\qed
\end{theorem}



\begin{remark}\label{rem:x11 cohomology}
We mentioned in Remark~\ref{rem:x11 case} that $\Qpq 21 \homeo \Xpq 1{}$ and promised to
compare our calculation of the cohomology here with the calculation in \cite{CHTFiniteProjSpace}.
As in the earlier remark, we write $\lambda$ for the tautological bundle over $\Xpq 1{}$,
so $\omega = \lambda^{\tensor 2}$. We will write $\cld$ for the Euler class of $\lambda\dual$.

Noticing that $\Qpq 21^\GG$ has only two components, our calculation here becomes
the following: 
$\HH^\gr(\Qpq 21)$ is generated as an algebra over $\HH^{RO(\Pi BU(1))}(B U(1))$ by
$\zeta_{00}$, $\zeta_{11}$, and $x_{1,0} \in \HH^{2\sigma}(\Qpq 21)$ such that
\begin{alignat*}{2}
    \cwd &\quad&&\text{is infinitely divisible by $\xi$,} \\
    \zeta_1 &&&\text{is invertible,} \\
    x_{1,0} &&&\text{is infinitely divisible by $\zeta_{00}$, and} \\
    (1-\kappa)x_{1,0} + e^2 &&&\text{is infinitely divisible by $\zeta_{11}$,} 
\end{alignat*}
with the relations
\begin{align*}
    \zeta_{00}\zeta_{11} &= \zeta_0 \\
    x_{1,0}^2 &= e^2 x_{1,0} \\
    \cwd (1-e^{-2}\kappa x_{1,0}) &= \tau(\iota^{-2})\zeta_1 x_{1,0}.
\end{align*}
Given these relations, we have that
\begin{align*}
    (1-e^{-2}\kappa x_{1,0})^2 &= 1 - 2e^{-2}\kappa x_{1,0} + 2e^{-4}\kappa x_{1,0}^2 \\
    &= 1 - 2e^{-2}\kappa x_{1,0} + 2e^{-2}\kappa x_{1,0} \\
    &= 1,
\end{align*}
so $1-e^{-2}\kappa x_{1,0}$ is a unit and its own inverse. This allows us to rewrite the last relation as
\[
    \cwd = \tau(\iota^{-2})\zeta_1 x_{1,0}(1-e^{-2}\kappa x_{1,0}) = \tau(\iota^{-2})\zeta_1 x_{1,0},
\]
so we do not need $\cwd$ as a generator in this case.

Under the identification $\Qpq 21 = \Xpq 1{}$, we have
\[
    \Qpq 21^{00} = \Xpq 1{}^0 = \Xp{} \qquad\text{and}\qquad \Qpq 21^{11} = \Xpq 1{}^1 = \Xq{},
\]
so $\zeta_{00} \in \HH^{\Omega_{00}}(\Qpq 21)$ corresponds to $\zeta_0\in \HH^{\Omega_0}(\Xpq 1{})$
(we write $\zeta_{00} \corrs \zeta_0$)
and $\zeta_{11} \in \HH^{\Omega_{11}}(\Qpq 21)$ corresponds to $\zeta_1\in \HH^{\Omega_1}(\Xpq 1{})$.
The element $x_{1,0}$ is represented by the point $\Qpq 21^{11}$ while
the point $\Xpq 1{}^1$ represents $\cld$, but to match up gradings
we need to have $x_{1,0} \corrs \zeta_0\cld$,
and we then see that
\[
    (1-\kappa)x_{1,0} + e^2 \corrs (1-\kappa)\zeta_0 \cld + e^2 = \zeta_1\cxld.
\]
We know that $\cld$ is infinitely divisible by $\zeta_0$, which corresponds to $x_{1,0}$ being
infinitely divisible by $\zeta_{00}$, and that $\cxld$ is infinitely divisible by $\zeta_1$,
which corresponds to $(1-\kappa)x_{1,0} + e^2$ being infinitely divisible by $\zeta_{11}$.
We also have
\begin{align*}
    x_{1,0}^2 &\corrs \zeta_0^2\cwd[2] \\
    &= \xi \cwd\cxwd + e^2\zeta_0\cwd \\
    &= e^2\zeta_0\cwd \\
    &\corrs e^2 x_{1,0}.
\end{align*}

As for $\cwd$, this is the Euler class of $O(2)$ on $\Xpq 1{}$, which we computed in \cite{CHTFiniteProjSpace}
as
\[
    e(O(2)) = \tau(\iota^{-2})\zeta_0\cld \corrs \tau(\iota^{-2})x_{1,0},
\]
so we again see that $\cwd = \tau(\iota^{-2})x_{1,0}$.
Being in the image of $\tau$, this element is indeed infinitely divisible by $\xi$
(because $\rho(\xi) = \iota^2$ is invertible).

Summarizing, our calculation here matches up with the calculation in \cite{CHTFiniteProjSpace}, with the
correspondence of elements being
\begin{align*}
    \zeta_{00} &\corrs \zeta_0 \\
    \zeta_{11} &\corrs \zeta_1 \\
    x_{1,0} &\corrs \zeta_0\cld \\
    x'_{1,0} = (1-\kappa)x_{1,0} + e^2 &\corrs \zeta_1\cxld \\
    \cwd = \tau(\iota^{-2})x_{1,0} &\corrs e(O(2)) = \tau(\iota^{-2})\zeta_0\cld.
\end{align*}
We also have that the unit $1-e^{-2}\kappa x_{1,0}$ corresponds to the unit
$1-e^{-2}\kappa\zeta_0\cld$ found in \cite[Proposition~11.5]{Co:InfinitePublished}.
\end{remark}

\section{Quadrics of type (D,D)}

We consider here the case of $\Qpq 2{2q}$ with $q > 1$. This entails the case $\Qpq{2q}2$, and
we will deal with $\Qpq 22$ in \S\ref{sec:22}.
The results in this section are very similar to those in the preceding section,
so we will skip over many of the proofs.

As in the preceding section and the introduction,
\[
    RO(\Pi\Qpq 2{2q}) = \Z\{1,\sigma,\Omega_{00},\Omega_{11},\Omega_1\}/\rels{\Omega_{00} + \Omega_{11} + \Omega_1 = 2\sigma - 2}
\]
and we have elements $\zeta_{00}$, $\zeta_{11}$, and $\zeta_1$
in the cohomology of $\Qpq 2{2q}$ corresponding to the three components of $\Qpq 2{2q}^\GG$.
We again consider two cofibrations, using the following maps, consistent with those in (\ref{eqn:inclusions}).
\begin{alignat*}{3}
    i_0\colon \Xpq 1q &\includesin \Qpq{2}{2q} &\qquad&& i[x:\vec y] &= [x:0:\vec y:\vec 0] \\
    i_1\colon \Xpq 1q &\includesin \Qpq{2}{2q} &&& i[x:\vec y] &= [x:0:\vec 0:\vec y] \\
    i_2\colon \Xpq 1q &\includesin \Qpq{2}{2q} &&& i[x:\vec y] &= [0:x:\vec y:\vec 0] \\
    i_3\colon \Xpq 1q &\includesin \Qpq{2}{2q} &&& i[x:\vec y] &= [0:x:\vec 0:\vec y].
\end{alignat*}
The normal bundle to each of these maps is isomorphic to
$\nu = \omega\dual \dirsum q\chi\omega\dual - O(2)$,
which has dimension $\omega + q\chi\omega\dual - 2$, which we also denote by $\nu$.
The following is proved the same way as Proposition~\ref{prop:db cofibration}.

\begin{proposition}\label{prop:dd cofibration}
We have cofibration sequences
\begin{align*}
    \Xpq 1q_+ &{}\xrightarrow{i_0} (\Qpq{2}{2q})_+ \to \susp^\nu i_3(\Xpq 1q)_+ \\
\intertext{and}
    \Xpq 1q_+ &{}\xrightarrow{i_2} (\Qpq{2}{2q})_+ \to \susp^\nu i_1(\Xpq 1q)_+
\end{align*}
\qed
\end{proposition}

Let
\begin{align*}
    x_{1,q} &= (i_3)_!(1) = [i_3(\Xpq 1q)]^* \in \HH^\nu(\Qpq 2{2q}) \\
\intertext{and}
    x'_{1,q} &= (i_1)_!(1) = [i_1(\Xpq 1q)]^* \in \HH^\nu(\Qpq 2{2q}).
\end{align*}
As in the (D,B) case, $x_{1,q}$ is infinitely divisible by $\zeta_{00}$ and
$x'_{1,q}$ is infinitely divisible by $\zeta_{11}$.

This time, we let
\[
    \divq = \cxwd[q] - e^{-2}\kappa\cxwd x_{1,q}.
\]
In \cite[Proposition~7.2]{CH:QuadricsII} 
we proved that
\begin{alignat*}{2}
    \cwd &\text{ is infinitely divisible by $\zeta_0$ and} \\
    \divq &\text{ is infinitely divisible by $\zeta_1$.}
\end{alignat*}

We then get the following in the same way as Proposition~\ref{prop:db splitting}.

\begin{proposition}\label{prop:dd splitting}
Let $q > 1$.
If $m \geq 0$, then there is a split short exact sequence
\begin{multline*}
    0 \to \susp^{\nu-m\Omega_0}\HH^{RO(\Pi BU(1))}(\Xpq 1q) \xrightarrow{(i_3)_!}
    \HH^{m\Omega_{11} + RO(\Pi BU(1))}(\Qpq 2{2q}) \\ \xrightarrow{i_0^*}
    \HH^{RO(\Pi BU(1))}(\Xpq 1q) \to 0.
\end{multline*}
If $m \leq 0$, then there is a split short exact sequence
\begin{multline*}
    0 \to \susp^\nu\HH^{RO(\Pi BU(1))}(\Xpq 1q) \xrightarrow{(i_1)_!}
    \HH^{m\Omega_{11} + RO(\Pi BU(1))}(\Qpq 2{2q}) \\ \xrightarrow{i_2^*}
    \susp^{-m\Omega_0}\HH^{RO(\Pi BU(1))}(\Xpq 1q) \to 0.
\end{multline*}
\qed
\end{proposition}

\begin{corollary}\label{cor:dd additive}
Additively,
\[
    \HH^\gr(\Qpq{2}{2q}) \iso \HH^\gr(\Xpq 1q) \dirsum \susp^\nu\HH^\gr(\Xpq 1q)
\]
where $\nu = \omega + q\chi\omega - 2$.
In gradings $m\Omega_{11} + RO(\Pi BU(1))$ with $m\geq 0$, 
the first summand is regraded via $i_0^*$ while the second is regraded via $i_3^*$;
if $m < 0$, the first summand is regraded via $i_2^*$ while the second is regraded via $i_1^*$.
\qed
\end{corollary}


The comments we made after Corollary~\ref{cor:db additive} about naming basis elements apply
here as well.

The relationship between $x'_{1,q}$ and $x_{1,q}$ is somewhat different in this case.
First note that the nonequivariant space underlying $\Qpq 2{2q}$ is a quadric of type D and that
\begin{align*}
    \rho(x_{1,q}) &= y \\
\intertext{but}
    \rho(x'_{1,q}) &= c^{q} - y,
\end{align*}
as the two are represented by the projectivizations of two
maximal isotropic affine subspaces whose intersection has codimension one in each.

\begin{lemma}\label{lem:dd xs}
$x'_{1,q} = -(1-\kappa)x_{1,q} + \zeta_1\divq = -(1-\epsilon)x_{1,q} + \zeta_1\cxwd[q]$
where $\epsilon = e^{-2}\kappa\zeta_0\cwd$ and $1-\epsilon$ is a unit in the cohomology of $BU(1)$.
\end{lemma}

\begin{proof}
That $1-\epsilon$ is a unit in the cohomology of $BU(1)$ was shown in
\cite[Proposition~11.5]{Co:InfinitePublished}.
The equality of the two expressions for $x_{1,q}'$ follows from the definition of $\divq$.

Taking a basis that is slightly different from our usual convention, we can write
\[
    x'_{1,q} = \alpha x_{1,q} + \beta \zeta_1\divq + \gamma e^2\cxwd[q-1] + \delta e^{-2}\kappa\zeta_0\cwd x_{1,q}
\]
for some $\alpha, \beta \in A(\GG)$ and $\gamma, \delta\in \Z$.
We have
\begin{align*}
    \rho(x'_{1,q}) &= c^q - y \\
    \rho(x_{1,q}) &= y \\
    \rho(\zeta_1\divq) &= c^q \\
    \rho(e^2\cxwd[q-1]) &= 0 \\
    \rho(e^{-2}\kappa\zeta_0\cwd x_{1,q}) &= 0,
\intertext{from which we see that $\rho(\alpha) = -1$ and $\rho(\beta) = 1$.
Taking fixed points, we have}
    (x'_{1,q})^\GG &= (1,0,y) \\
    x_{1,q}^\GG &= (0,1,y) \\
    (\zeta_1\divq)^\GG &= (1,1,0)\bigl[ (1,1,c^q) - 2(0,1,cy) \bigr] \\ &= (1,-1,0) \\
    (e^2\cxwd[q-1])^\GG &= (1,1,c^{q-1}) \\
    (e^{-2}\kappa\zeta_0\cwd x_{1,q})^\GG &= 2(0,0,c^{q-1}).
\end{align*}
This tells us that $\alpha^\GG = 1$, $\beta^\GG = 1$, and $\gamma = \delta = 0$.
Hence $\alpha = -(1-\kappa)$ and $\beta = 1$, giving the formula in the lemma.
\end{proof}

The point, as in the preceding section, is that we don't really need $x'_{1,q}$ as a generator,
since it can be written in terms of $x_{1,q}$. Since $1-\kappa$ is its own inverse, we can also write
\[
    -(1-\kappa)x'_{1,q} = x_{1,q} - (1-\kappa)\zeta_1\divq,
\]
which is the expression we use in the theorem below.

The multiplicative structure is as follows, proven similarly to
Theorem~\ref{thm:db multiplicative},
using \cite[Theorem~7.6]{CH:QuadricsII}.

\begin{theorem}\label{thm:dd multiplicative}
Let $q > 1$.
As an algebra over $\HH^{RO(\Pi BU(1))}(B U(1))$, 
the cohomology $\HH^\gr(\Qpq{2}{2q})$
is generated by
\begin{align*}
    \zeta_{00} &\in \HH^{\Omega_{00}}(\Qpq 2{2q}) \\
    \zeta_{11} &\in \HH^{\Omega_{11}}(\Qpq 2{2q}) \\
    x_{1,q} &\in \HH^{(q-1)\chi\omega + 2\sigma}(\Qpq{2}{2q})
\end{align*}
subject to the facts that
\begin{alignat*}{2}
    &\cwd &\quad&\text{is infinitely divisible by $\zeta_0$,} \\
    &\divq = \cxwd[q] - e^{-2}\kappa  \cxwd x_{1,q} &&\text{is infinitely divisible by $\zeta_1$,} \\
    &x_{1,q} &&\text{is infinitely divisible by $\zeta_{00}$, and} \\
    &x_{1,q} - (1-\kappa)\zeta_1\divq &&\text{is infinitely divisible by $\zeta_{11}$,} 
\end{alignat*}
and the following relations:
\begin{align*}
    \zeta_{00}\zeta_{11} &= \zeta_0 \\
    x_{1,q}^2 &=
        \begin{cases}
            e^2\cxwd[q-1]x_{1,q} & \text{if $q$ is odd} \\
            \zeta_1\cxwd[q] x_{1,q} & \text{if $q$ is even,}
        \end{cases}
    \\
   \cwd\divq &= \tau(\iota^{-2})\zeta_0\cwd x_{1,q}.
\end{align*}
\qed
\end{theorem}

\section{A Grassmannian}

We pause from our calculations to
look at an interesting special case in detail, which we will use in the next section:
The Grassmannian $\Grpq 222$ of $2$-planes in $\Cpq 22$ is $\GG$-diffeomorphic to
$\Qexpq 24$. 
To see this, we use a modified version of the Pl\"ucker embedding.
We represent 2-planes in $\Cpq 22$ by $2\times 4$-matrices of complex numbers, which we
think of as arrays $[\vec a_1\ \vec a_2\ \vec a_3\ \vec a_4]$ with linearly independent rows, in which
the columns
$\vec a_1$ and $\vec a_2$ lie in $\Cp 2$ while $\vec a_3$ and $\vec a_4$ lie in $\Cq 2$.
Such a matrix represents the 2-plane in $\Cpq 22$ spanned by its rows, and two such matrices represent the
same 2-plane if they are in the same orbit of the general linear group $\Gl(\Cp 2)$.
The Pl\"ucker embedding (modified slighly by a negative sign) is then the map
$k\colon \Grpq 222\to \Xpq 24$ defined by
\[
    k[\vec a_1\ \vec a_2\ \vec a_3\ \vec a_4]
    = [\Delta_{12}:\Delta_{34}:\Delta_{13}:\Delta_{14}:\Delta_{23}:-\Delta_{24}],
\]
where $\Delta_{ij} = |\vec a_i\ \vec a_j|$ is the determinant.
The image of this map is exactly $\Qexpq 24$.
The fixed set $\Grpq 222^\GG$ consists of three components, which correspond to
the components of $\Qexpq 24^\GG$ as follows:
\begin{align*}
    k(\Grp 22) &= \Qpq 24^{00} = [1:0:0:0:0:0] \\
    k(\Xp 2\times \Xq 2) &= \Qpq 24^{1} = \Qexq 4 \\
    k(\Grq 22) &= \Qpq 24^{11} = [0:1:0:0:0:0]
\end{align*}
We will use this to identify
\[
    RO(\Pi\Grpq 222) = \Z\{1,\sigma,\Omega_{00},\Omega_1,\Omega_{11}\}/\rels{\Omega_{00} + \Omega_1 + \Omega_{11} = 2\sigma - 2}.
\]
Via the inclusion $\Grpq 222 \subset \Grpq 2\infty\infty = BU(2)$, we also can identify
\[
    RO(\Pi\Grpq 222) \iso RO(\Pi BU(2)).
\]
The latter was used in \cite{CH:bu2};
what we there called $\Omega_0$ we are here calling $\Omega_{00}$ and what we there called
$\Omega_2$ we are here calling $\Omega_{11}$, while $\Omega_1$ means the same thing here as it did there.

The Grassmannian has a tautological 2-plane bundle $\pi$.
The Pl\"ucker embedding classifies its exterior power, or determinant line bundle, $\lambda = \ext^2\pi$,
that is, $\lambda = k^*\omega$.
Writing $\cld$ for the Euler class of the dual of $\lambda$, this gives us the following
corollary of Theorem~\ref{thm:dd multiplicative} (combined with the structure of the cohomology of
$BU(1)$ from \cite{Co:InfinitePublished}).

\begin{corollary}\label{cor:gr24}
As an algebra over $\HS$, $\HH^\gr(\Grpq 222)$ is generated by elements
\begin{alignat*}{2}
    & \zeta_{00} &\quad& \text{in grading $\Omega_{00}$} \\
    & \zeta_1 && \text{in grading $\Omega_1$} \\
    & \zeta_{11} && \text{in grading $\Omega_{11}$} \\
    & \cld && \text{in grading $\lambda = 2 + \Omega_1$} \\
    & \cxld && \text{in grading $\chi\lambda = 2 + \Omega_{00} + \Omega_{11}$} \\
    & x_{1,2} && \text{in grading $\lambda + 2\chi\lambda -2 = \chi\lambda + 2\sigma$}
\end{alignat*}
subject to the facts that
\begin{alignat*}{2}
    & \cld &\quad& \text{is infinitely divisible by $\zeta_0 = \zeta_{00}\zeta_{11}$} \\
    & \divq = \cxld[2] - e^{-2}\kappa \cxld x_{1,2} && \text{is infinitely divisible by $\zeta_1$} \\
    & x_{1,2} && \text{is infinitely divisible by $\zeta_{00}$, and } \\
    & x_{1,2} - (1-\kappa)\zeta_1\divq && \text{is infinitely divisible by $\zeta_{11}$}
\end{alignat*}
and the following relations:
\begin{align*}
    \zeta_{00}\zeta_{11}\zeta_1 &= \xi \\
    \zeta_1\cxld &= (1-\kappa)\zeta_{00}\zeta_{11}\cld + e^2 \\
    x_{1,2}^2 &= \zeta_1\cxld[2] x_{1,2} \\
    \cld\divq &= \tau(\iota^{-2})\zeta_{00}\zeta_{11}\cld x_{1,2}
\end{align*}
\qed
\end{corollary}

\begin{remark}
What submanifold of $\Grpq 222$ represents $x_{1,2}$?
We know that, considered as an element of the cohomology
of $\Qpq 24$, it is represented by $\Xpq 12 \subset \Qpq 24$.
Examining the Pl\"ucker embedding, $\Xpq 12$ is the image of
\[
    \Grpq 212\subset \Grpq 222,
\]
where $\Cpq 12\subset \Cpq 22$ is missing the first copy of $\C$.
This is the zero locus of the section of $\pi\dual$ given by
\[
    \pi \includesin \Cpq 22 \to \C,
\]
where the second map is projection to the first factor of $\C$. 
Thus, $\Grpq 212$ represents both $x_{1,2}$ and $\cgd = e(\pi\dual)$
(the class we called $\cwd$ in \cite{CH:bu2}),
with a caveat: These classes live in different gradings, so we are considering
$\Grpq 212$ as having two different equivariant dimensions in representing these elements.
The gradings of the two elements, hence codimensions of $\Grpq 212$ in $\Grpq 222$, are
\begin{align*}
    \grad x_{1,2} &= \chi\omega + 2\sigma = 2 + 2\sigma + \Omega_{00} + \Omega_{11} \\
\intertext{and}
    \grad \cgd &= \pi = 4 + \Omega_1 + 2\Omega_{11} = 2 + 2\sigma - \Omega_{00} + \Omega_{11}.
\end{align*}
The difference is $\grad x_{1,2} - \grad \cgd = 2\Omega_{00}$, which we can accomodate
because $\Grpq 212$ has no fixed set component over $\Grp 22 = \Qpq 24^{00}$.
That is, geometrically, we are free to adjust its dimension by multiples of $\Omega_{00}$,
and algebraically, any element it represents will be infinitely divisible by $\zeta_{00}$.
The upshot is that
\[
    x_{1,2} = \zeta_{00}^2 \cgd \qquad\text{or}\qquad \cgd = \zeta_{00}^{-2}x_{1,2}.
\]
So we could rewrite Corollary~\ref{cor:gr24} using $\cgd$ in place of $x_{1,2}$
to better compare the result to the cohomology of $BU(2)$ computed in \cite{CH:bu2}.
\end{remark}

\begin{remark}
We can also ask what submanifold of $\Grpq 222$ represents $x'_{1,2}$, the class
represented in the cohomology of $\Qpq 24$ by 
\[
    i_1(\Xpq 12) = \{ [\Delta_{12}:\Delta_{34}:\Delta_{13}:\Delta_{14}:\Delta_{23}:-\Delta_{24}]
        \mid \Delta_{34} = \Delta_{13} = \Delta_{14} = 0 \}.
\]
This is the image of 
$\Xp{}\times \Xpq 12 \subset \Grpq 222$, where the copy of $\C$ in $\Xp{}$
is the second copy of $\C$ in $\Cpq 22$.
\end{remark}

\section{27 lines}\label{sec:27 lines}

As we did in \cite{CH:QuadricsI}, we consider a
well-known example in enumerative geometry, the result that
a cubic surface in $\Xp 4$ contains 27 lines;
see, for example \cite[\S6.2.1]{3264EisenbudHarris}.
In \cite{Braz:equivenumerative}, Brazelton determined the structure of that
set of 27 lines as an $S_4$-set when the symmetric group $S_4$ acts on $\Cp 4$ as its regular representation,
which then implies results for the subgroups of $S_4$.
Here, as an application of our results,
we use the calculation of the cohomology of $\Grpq 222$ in the preceding section to not only reproduce one of
Brazelton's calculations in the case of $\GG$, but to determine a bit more information about
those 27 lines.

In \cite{CH:QuadricsI}, we considered the case of $\Xpq 3{}$, while here
we consider the question of how many lines lie on a smooth $\GG$-invariant cubic surface in
$\Xpq 22$.
Such a surface is the zero locus of a homogeneous cubic polynomial $f(a_1,a_2,b_1,b_2)$
with $(a_1,a_2,b_1,b_2)\in \Cpq 22$.
For the zero locus to be $\GG$-invariant, we want either $tf = f$ or $tf = -f$.
If $tf = f$, we say that $f$ is \emph{even}, and if $tf = -f$, we say that $f$ is \emph{odd}.
(These really refer to the total degree in the variables $b_1$ and $b_2$ of each monomial in $f$.)

Consider first an even cubic $f$.
Such a polynomial is a $\GG$-equivariant linear functional
\[
    f\colon \Sym^3(\Cpq 22) \to \C
\]
and, on restricting to 2-planes, we get a section of $\Sym^3(\pi\dual)$ over $\Grpq 222$.
Each point in the zero locus of this section gives us a line on the cubic surface defined by $f$,
and this zero locus represents the Euler class of the bundle.
Hence, determining the set of such lines comes down to computing the Euler class
$e(\Sym^3(\pi\dual))$ in the cohomology of $\Grpq 222$
and determining the $\GG$-set in $\Grpq 222$ that represents it.

On the other hand, an odd cubic polynomial gives a linear functional
\[
    \Sym^3(\Cpq 22) \to \Cq{},
\]
hence a section of $\Sym^3(\chi\pi\dual)$,
and the question then comes down to computing the Euler class
$e(\Sym^3(\chi\pi\dual))$.

\subsection{The even case}
To compute $e(\Sym^3(\pi\dual))$, we first find its grading, which is determined
by the representations of $\GG$ given by the fibers of $\Sym^3(\pi\dual)$ 
over each of the components of $\Grpq 222^\GG$.
Over $\Grp 22$ the fibers are
\[
    \Sym^3(\Cp 2) \iso \Cp 4,
\]
over $\Grq 22$ the fibers are
\[
    \Sym^3(\Cq 2) \iso \Cq 4,
\]
and over $\Xp 2\times \Xq 2$ the fibers are
\[
    \Sym^3(\Cpq 1{}) \iso \Cpq 22.
\]
This gives
\[
    \grad \Sym^3(\pi\dual) = 4\Omega_{11} + 2\Omega_1 + 8.
\]
The identification of $\Grpq 222$ with $\Qpq 24$ and Corollary~\ref{cor:dd additive}
show that, in gradings $4\Omega_{11} + 2\Omega_1 + RO(\GG)$,
the cohomology of $\Grpq 222$ has a basis over $\HS$ given by the following elements.
\begingroup
\allowdisplaybreaks[0]
\begin{alignat*}{2}
    \text{\textbf{Element}} &\qquad& &\text{\textbf{Grading }} 4\Omega_{11} + 2\Omega_1 + {} \\
    \zeta_{11}^4\zeta_1^2 &&& 0 \\
    \zeta_{00}^{-1}\zeta_{11}^3\cld &&& 4-2\sigma \\
    \zeta_{00}^{-2}\zeta_{11}^2\cld\cxld &&& 6-2\sigma \\
    \zeta_{00}^{-3}\zeta_{11} x_{1,2} &&& 6 - 2\sigma \\
    \zeta_{00}^{-4}\cxld x_{1,2} &&& 8 - 2\sigma \\
    \zeta_{00}^{-3}\zeta_{11} \cld\cxld x_{1,2} &&& 8
\end{alignat*}
\endgroup
Of these, only the last two can contribute to the group in grading $4\Omega_{11} + 2\Omega_1 + 8$,
so
\[
    e(\Sym^3(\pi\dual)) = \alpha \zeta_{00}^{-3}\zeta_{11} \cld\cxld x_{1,2} 
    + \beta e^2 \zeta_{00}^{-4}\cxld x_{1,2}
\]
for some $\alpha\in A(\GG)$ and $\beta\in \Z$.

To determine $\alpha$ and $\beta$, we first reduce by $\rho$, using the known nonequivariant calculation
(which follows easily from the splitting principle):
\begin{align*}
    \rho(e(\Sym^3(\pi\dual))) &= 27 c^2 y \\
    \rho(\zeta_{00}^{-3}\zeta_{11} \cld\cxld x_{1,2}) &= c^2 y \\
    \rho(e^2 \zeta_{00}^{-4}\cxld x_{1,2}) &= 0
\end{align*}
Hence $\rho(\alpha) = 27$.

Now we want to compute fixed points. For $\Sym^3(\pi\dual)$, its restriction
to the point $\Grp 22$ is $\Cp 4$, hence its fixed point bundle there is also $\Cp 4$ and has 0 Euler characteristic.
Its restriction to the point $\Grq 22$ is $\Cq 4$, hence its fixed point bundle is the 0 bundle,
which has Euler characteristic 1.
Over the third component, $\Xp 2\times \Xq 2$, $\pi\dual$ splits as the sum of the duals of the two tautological
line bundles over these projective spaces, which we'll write as $\pi\dual = O_1(1)\dirsum O_2(1)$, hence
\begin{multline*}
    \Sym^3(\pi\dual)|(\Xp 2\times\Xq 2) \iso \\
    O_1(3) \dirsum ( O_1(2)\tensor O_2(1)) \dirsum (O_1(1) \tensor O_2(2))
    \dirsum O_2(3)
\end{multline*}
Thus, the fixed-point bundle is
\[
    \Sym^3(\pi\dual)^\GG|(\Xp 2\times\Xq 2) =
    O_1(3) \dirsum (O_1(1) \tensor O_2(2)).
\]
Writing $x_1$ and $x_2$ for the nonequivariant Euler classes of $O_1(1)$ and $O_2(1)$, respectively, we then get
(writing components in the order $(00,11,01)$)
\[
    e(\Sym^3(\pi\dual))^\GG = (0,1,3x_1(x_1 + 2x_2)) = (0,1, 6x_1 x_2)
\]
because $x_1^2 = 0$. The fixed sets of the basis elements are
\begin{align*}
    (\zeta_{00}^{-3}\zeta_{11} \cld\cxld x_{1,2})^\GG &= (0,0,(x_1+x_2)x_2) = (0,0,x_1x_2) \\
    (e^2 \zeta_{00}^{-4}\cxld x_{1,2})^\GG &= (0,1,(x_1+x_2)x_2) = (0,1,x_1x_2).
\end{align*}
From this we get that $\beta = 1$ and $\alpha^\GG = 5$.
From $\rho(\alpha) = 27$ and $\alpha^\GG = 5$ we conclude that $\alpha = 5 + 11g$, hence
\[
    e(\Sym^3(\pi\dual)) = (5 + 11g) \zeta_{00}^{-3}\zeta_{11} \cld\cxld x_{1,2} 
    + e^2 \zeta_{00}^{-4}\cxld x_{1,2}
\]

These multiples of $x_{1,2}$ are pushforwards along $i_3$ of elements from the cohomology of $\Xpq 12$,
and in that cohomology we have
\begin{align*}
    (5+11g)\zeta_0\cld\cxld + e^2\cxld
    &= (6 + 10g)\zeta_0\cld\cxld + ((1-\kappa)\zeta_0\cld + e^2)\cxld \\
    &= (6 + 10g)\zeta_0\cld\cxld + \zeta_1\cxld[2] \\
    &= (6 + 10g)[\Xq{}]^* + [\Xp{}]^*
\end{align*}
Pushing forward into $\Qpq 24$ and then interpreting in terms of $\Grpq 222$, we get
\begin{align*}
    e(\Sym^3(\pi\dual))
    &= (6 + 10g)[\Grpq 21{}]^* + [\Grq 22]^* \\
    &= 10[\GG]^* + 6[\Grpq 21{}]^* + [\Grq 22]^*.
\end{align*}
Thus, the 27 lines appear as 10 free pairs of the form $\GG\times\PP^1$, 6 invariant lines of
the form $\Xpq 1{}$, and one invariant line of the form $\Xq 2$.

To say that an even cubic in $\Xpq 22$ must contain a line of the form $\Xq 2$
is to say that it must contain the unique line of that form, $\Xq 2\subset \Xpq 22$.
We can also see this directly: If $f(a_1,a_2,b_1,b_2)$ is an even cubic,
then each of its monomials must be a multiple of either $a_1$ or $a_2$, hence $f$ must
vanish on all of $\Xq 2$, where $a_1 = a_2 = 0$.

\begin{remark}
This improves somewhat on a calculation done by Brazelton in \cite{Braz:equivenumerative},
for the action of the subgroup he calls $C_2^e$, which is $\{e, (1\ 2)(3\ 4)\} < S_4$.
Note that his action of $C_2^e$ on $\Cp 4$ interchanges the first two summands and also
the last two summands. A simple change of basis identifies this with the $\GG$ representation
$\Cpq 22$.

In his Table~1, Brazelton shows the resulting $C_2^e$-set of lines on a cubic as
$10[\GG/e] + 7[\GG/\GG]$ and does not differentiate the two kinds of lines
that occur among the seven left invariant by $\GG$.
We compared our approach to Brazelton's in detail in \cite[\S 5.1]{CH:QuadricsI}
and that discussion applies here as well.
Essentially, we are computing an element in $H^\GG_0(\Grpq 222)$ while
Brazelton's calculation amounts to computing its image in $A(\GG)$ resulting from the projection
$\Grpq 222\to *$ to a point.
\end{remark}

\subsection{The odd case}
We could do a similar analysis as above to calculate and interpret
$e(\Sym^3(\chi\pi\dual))$.
A quicker approach is to take advantage of the $\GG$-auto\-morph\-ism
\[
    \chi\colon \Grpq 222 \to \Grpq 222
\]
that interchanges $\Cp 2$ and $\Cq 2$, taking $\pi\dual$ to $\chi\pi\dual$.
Applying this to the result in the even case, we get
\[
    e(\Sym^3(\chi\pi\dual)) = 10[\GG]^* + 6[\Grpq 21{}]^* + [\Grp 22]^*.
\]
So we get a result similar to the even case, except that one of the $\GG$-invariant lines
now has the form $\Grp 22$ rather than $\Grq 22$.
Thus, an odd cubic in $\Xpq 22$ must contain the line $\Xp 2 \subset \Xpq 22$.
Again, we can see this latter fact directly:
If $f(a_1,a_2,b_1,b_2)$ is an odd cubic, then each of its monomials
must be a multiple of $b_1$ or $b_2$, hence $f$ must vanish on all of $\Xp 2$.

\section{The product of two projective lines}\label{sec:22}

Finally, we look at the case of $\Qpq 22$, which we noted in \S\ref{sec:notations}
is diffeomorphic to $\Xpq 1{}\times \Xpq 1{}$
and is the only case with four components to its fixed set.
In fact, its fixed set consists of four points, and recall that we label them as follows,
to be consistent with the labeling of the fixed points of $\Xpq 1{}\times\Xpq 1{}\subset B T(2)$
used in \cite{CH:bt2}:
\begin{align*}
    \Qpq 22^{00} &= [1:0:0:0] \\
    \Qpq 22^{11} &= [0:1:0:0] \\
    \Qpq 22^{01} &= [0:0:1:0] \\
    \Qpq 22^{10} &= [0:0:0:1].
\end{align*}
Recall also that, with this labeling, we write
\[
    RO(\Pi\Qpq 22) = \Z\{1,\sigma,\Omega_{00},\Omega_{11},\Omega_{01},\Omega_{10}\}/
        \rels{\textstyle \sum_{i,j} \Omega_{ij} = 2\sigma - 2},
\]
and the inclusion $\Qpq 22\includesin \Xpq 22 \includesin B U(1)$
induces the map on representation rings $RO(\Pi B U(1)) \to RO(\Pi\Qpq 22)$ given by
\[
    \Omega_0 \mapsto \Omega_{00} + \Omega_{11} \qquad\text{and}\qquad
    \Omega_1 \mapsto \Omega_{01} + \Omega_{10}.
\]
In cohomology, we have $\zeta_0\mapsto \zeta_{00}\zeta_{11}$, and we will write
\[
    \zeta_0 = \zeta_{00}\zeta_{11}\in \HH^{\Omega_{00}+\Omega_{11}}(\Qpq 22).
\]
Similarly, we write $\zeta_1 = \zeta_{01}\zeta_{10}$.

Considering $RO(\Pi B U(1))$ as a subgroup of $RO(\Pi\Qpq 22)$, its group of cosets is
\[
    RO(\Pi\Qpq 22)/RO(\Pi B U(1)) \iso \Z\{\Omega_{11},\Omega_{10} \}.
\]
Of course, which elements we use as representatives is somewhat arbitrary, but it will help to
pick two as a basis and use them consistently, and this choice works well for what follows.

We consider the four inclusions of $\Xpq 1{}$ in $\Qpq 22$ given in (\ref{eqn:inclusions}),
which we repeat here:
\begin{alignat*}{3}
    i_0&\colon \Xpq 1{}\includesin \Qpq 22 &\qquad&& i_0[x:y] &= [x:0:y:0] \\
    i_1&\colon \Xpq 1{}\includesin \Qpq 22 &&& i_1[x:y] &= [x:0:0:y] \\
    i_2&\colon \Xpq 1{}\includesin \Qpq 22 &&& i_2[x:y] &= [0:x:y:0] \\
    i_3&\colon \Xpq 1{}\includesin \Qpq 22 &&& i_3[x:y] &= [0:x:0:y]
\end{alignat*}
Given which fixed points they hit and which they miss, the induced maps in cohomology satisfy the following.
\begin{align*}
    i_0^*(\zeta_{00}) &= \zeta_0
        & i_0^*(\zeta_{11}) &= 1 
        & i_0^*(\zeta_{01}) &= \zeta_1
        & i_0^*(\zeta_{10}) &= 1 \\
    i_1^*(\zeta_{00}) &= \zeta_0
        & i_1^*(\zeta_{11}) &= 1 
        & i_1^*(\zeta_{01}) &= 1
        & i_1^*(\zeta_{10}) &= \zeta_1 \\
    i_2^*(\zeta_{00}) &= 1 
        & i_2^*(\zeta_{11}) &= \zeta_0
        & i_2^*(\zeta_{01}) &= \zeta_1
        & i_2^*(\zeta_{10}) &= 1 \\
    i_3^*(\zeta_{00}) &= 1 
        & i_3^*(\zeta_{11}) &= \zeta_0
        & i_3^*(\zeta_{01}) &= 1
        & i_3^*(\zeta_{10}) &= \zeta_1
\end{align*}

The four inclusions give four cofibration sequences.

\begin{proposition}\label{prop:22 cofibration}
We have the following four cofibration sequences.
\begin{align*}
    \Xpq 1{}_+ \xrightarrow{i_0} {}&(\Qpq 22)_+ \to \Susp^{2\sigma} i_3(\Xpq 1{})_+ \\
    \Xpq 1{}_+ \xrightarrow{i_1} {}&(\Qpq 22)_+ \to \Susp^{2\sigma} i_2(\Xpq 1{})_+ \\
    \Xpq 1{}_+ \xrightarrow{i_2} {}&(\Qpq 22)_+ \to \Susp^{2\sigma} i_1(\Xpq 1{})_+ \\
    \Xpq 1{}_+ \xrightarrow{i_3} {}&(\Qpq 22)_+ \to \Susp^{2\sigma} i_0(\Xpq 1{})_+ 
\end{align*}
\end{proposition}

\begin{proof}
The projection
\[
    \Qpq 22 \setminus i_0(\Xpq 1{}) \to i_3(\Xpq 1{}), \quad
    [x:u:y:v] \mapsto [0:u:0:v]
\]
can be identified with the normal bundle to $i_3(\Xpq 1{})$, which has equivariant dimension $2\sigma$.
The other sequences are similar.
\end{proof}

This gives us four long exact sequences and we will show that each breaks into short exact sequences
in a range of gradings.
Note that \cite[Proposition~7.2]{CH:QuadricsII} and the remark following it show that $\cwd$ is infinitely divisible by $\zeta_0$
and $\cxwd$ is infinitely divisible by $\zeta_1$.

\begin{proposition}\label{prop:22 splitting}
Consider gradings $m\Omega_{11} + n\Omega_{10} + RO(\Pi B U(1))$ with $m,n\in\Z$.
If $m, n\geq 0$, we have a split short exact sequence
\begin{multline*}
    0 \to
    \susp^{2\sigma-m\Omega_0-n\Omega_1}\HH^{RO(\Pi B U(1))}(\Xpq 1{}) \\ \xrightarrow{(i_3)_!}
    \HH^{m\Omega_{11} + n\Omega_{10} + RO(\Pi B U(1))}(\Qpq 22) \\ \xrightarrow{i_0^*}
    \HH^{RO(\Pi B U(1))}(\Xpq 1{})
    \to 0
\end{multline*}
If $m\geq 0$ and $n\leq 0$, we have a split short exact sequence 
\begin{multline*}
    0 \to
    \susp^{2\sigma-m\Omega_0}\HH^{RO(\Pi B U(1))}(\Xpq 1{}) \\ \xrightarrow{(i_2)_!}
    \HH^{m\Omega_{11} + n\Omega_{10} + RO(\Pi B U(1))}(\Qpq 22) \\ \xrightarrow{i_1^*}
    \susp^{-n\Omega_1}\HH^{RO(\Pi B U(1))}(\Xpq 1{})
    \to 0
\end{multline*}
If $m\leq 0$ and $n\geq 0$, we have a split short exact sequence 
\begin{multline*}
    0 \to
    \susp^{2\sigma-n\Omega_1}\HH^{RO(\Pi B U(1))}(\Xpq 1{}) \\ \xrightarrow{(i_1)_!}
    \HH^{m\Omega_{11} + n\Omega_{10} + RO(\Pi B U(1))}(\Qpq 22) \\ \xrightarrow{i_2^*}
    \susp^{-m\Omega_0}\HH^{RO(\Pi B U(1))}(\Xpq 1{})
    \to 0
\end{multline*}
If $m\leq 0$ and $n\leq 0$, we have a split short exact sequence 
\begin{multline*}
    0 \to
    \susp^{2\sigma}\HH^{RO(\Pi B U(1))}(\Xpq 1{}) \\ \xrightarrow{(i_0)_!}
    \HH^{m\Omega_{11} + n\Omega_{10} + RO(\Pi B U(1))}(\Qpq 22) \\ \xrightarrow{i_3^*}
    \susp^{-m\Omega_0-n\Omega_1}\HH^{RO(\Pi B U(1))}(\Xpq 1{})
    \to 0
\end{multline*}
\end{proposition}

\begin{proof}
The notation used in the statement of the proposition means, in the first case, that,
if $\alpha\in RO(\Pi BU(1))$, then we have a short exact sequence
\begin{multline*}
    0 \to
    \HH^{-2\sigma + m\Omega_0 + n\Omega_1 + \alpha}(\Xpq 1{})  \xrightarrow{(i_3)_!}
    \HH^{m\Omega_{11} + n\Omega_{10} + \alpha}(\Qpq 22) \\ \xrightarrow{i_0^*}
    \HH^{\alpha}(\Xpq 1{})
    \to 0
\end{multline*}
That the exact sequence has these gradings follows from the fact that
$i_0^*(\Omega_{11}) = 0$, $i_0^*(\Omega_{10}) = 0$,
$i_3^*(\Omega_{11}) = \Omega_0$, and $i_3^*(\Omega_{10}) = \Omega_1$.
The gradings are as shown in the other cases for similar reasons.

When $m = n = 0$, so we are
in the $RO(\Pi B U(1))$ grading, the fact that $\cwd$ is divisible by $\zeta_0$ and $\cxwd$ is divisible by $\zeta_1$
shows that $i_0^*$ is surjective, as elements in a basis of the cohomology of $\Xpq 1{}$ are hit by
elements of the same names in the cohomology of $\Qpq 22$. The same is true for restriction along the other three
inclusions. (This is a special case of \cite[Proposition~7.4]{CH:QuadricsII}.)
Because $i_0^*(\zeta_{11}) = 1$ and $i_0^*(\zeta_{10}) = 1$, if $m,n\geq 0$
we can hit basis elements by multiplying basis elements in the
$RO(\Pi B U(1))$ grading by $\zeta_{11}^m\zeta_{10}^n$. Hence, in this range of gradings,
$i_0^*$ is surjective, and the resulting short exact sequence splits because the cohomology
of $\Xpq 1{}$ is free.

In the second case, where $m\geq 0$ and $n\leq 0$,
$i_1^*$ is surjective because we can take elements in the $RO(\Pi BU(1))$ grading
and multiply by $\zeta_{11}^m\zeta_{01}^{-n}$ to get elements that map to basis elements
in the desired grading,
using that $i_1^*(\zeta_{11}) = 1$ and $i_1^*(\zeta_{01}) = 1$, and that
$\Omega_{01} + \Omega_{10} \in RO(\Pi BU(1))$.

The other two cases are similar. 
\end{proof}

\begin{corollary}\label{cor:22 additive}
Additively, we have the following direct sum decompositions.
\begin{multline*}
    \HH^{m\Omega_{11} + n\Omega_{10} + RO(\Pi BU(1))}(\Qpq 22) \\
    \iso
    \begin{cases}
        \HH^{RO(\Pi B U(1))}(\Xpq 1{}) \dirsum \susp^{2\sigma-m\Omega_0-n\Omega_1}\HH^{RO(\Pi B U(1))}(\Xpq 1{}) \\
            \qquad\text{if $m, n\geq 0$,} \\
        \susp^{-n\Omega_1}\HH^{RO(\Pi B U(1))}(\Xpq 1{}) \dirsum \susp^{2\sigma-m\Omega_0}\HH^{RO(\Pi B U(1))}(\Xpq 1{}) \\
            \qquad\text{if $m\leq 0$ and $n\leq 0$,} \\
        \susp^{-m\Omega_0}\HH^{RO(\Pi B U(1))}(\Xpq 1{}) \dirsum \susp^{2\sigma-n\Omega_1}\HH^{RO(\Pi B U(1))}(\Xpq 1{}) \\
            \qquad\text{if $m\geq 0$ and $n\geq 0$,} \\
        \susp^{-m\Omega_0-n\Omega_1}\HH^{RO(\Pi B U(1))}(\Xpq 1{}) \dirsum \susp^{2\sigma}\HH^{RO(\Pi B U(1))}(\Xpq 1{}) \\
            \qquad\text{if $m,n\leq 0$.}
    \end{cases}
\end{multline*}
\qed
\end{corollary}

To write down bases using Proposition~\ref{prop:22 splitting}, on the face of it we should need to use multiples of
the appropriate one of $(i_3)_!(1) = x_{1,1}$ 
(using the notation of \cite{CH:QuadricsII}), 
$(i_2)_!(1)$, $(i_1)_!(1)$, or $(i_0)_!(1)$.
Unlike we did earlier in the cases of $\Qpq 2n$ for $n>2$, we will not give the latter three temporary names
but instead immediately observe the following.

\begin{lemma}\label{lem:22 xs}
\begin{align*}
    (i_2)_!(1) &= -(1-\kappa)(1-\epsilon)x_{1,1} + \zeta_0\cwd \\
    (i_1)_!(1) &= -(1-\epsilon)x_{1,1} + \zeta_1\cxwd \\
    (i_0)_!(1) &= (1-\kappa)x_{1,1} + e^2
\end{align*}
where $\epsilon = e^{-2}\kappa \zeta_0\cwd$ and $1-\epsilon$ is one of the units in the cohomology
of $B U(1)$ found in \cite[Proposition~11.5]{Co:InfinitePublished}.
\end{lemma}

\begin{proof}
All of these elements live in grading $2\sigma \in RO(\GG)$, and we can use as a basis of
$\HH^{RO(\GG)}(\Qpq 22)$ the set
\[
    \{ 1, \zeta_0\cwd, x_{1,1}, \zeta_0\cwd x_{1,1} \}.
\]
From this, we can write
\[
    (i_2)_!(1) = \alpha x_{1,1} + \beta \zeta_0\cwd + \gamma e^2 + \delta e^{-2}\kappa \zeta_0\cwd x_{1,1},
\]
for some $\alpha,\beta \in A(\GG)$ and $\gamma,\delta\in\Z$.
Applying $\rho$, we have
\begin{align*}
    \rho((i_2)_!(1)) &= c - y \\
    \rho(x_{1,1}) &= y \\
    \rho(\zeta_0\cwd) &= c \\
    \rho(e^2) &= 0 \\
    \rho(e^{-2}\kappa\zeta_0\cwd x_{1,1}) &= 0,
\end{align*}
so $\rho(\alpha) = -1$ and $\rho(\beta) = 1$.
Taking fixed points (and listing the components in the order $00$, $11$, $01$, $10$), we have
\begin{align*}
    (i_2)_!(1)^\GG &= (0,1,1,0) \\
    x_{1,1}^\GG &= (0,1,0,1) \\
    (\zeta_0\cwd)^\GG &= (0,0,1,1) \\
    (e^2)^\GG &= (1,1,1,1) \\
    (e^{-2}\kappa\zeta_0\cwd x_{1,1}) &= (0,0,0,2).
\end{align*}
This gives us $\alpha^\GG = 1$, $\beta^\GG = 1$, $\gamma = 0$, and $\delta = -1$.
Thus, $\alpha = -(1-\kappa)$ and $\beta = 1$, and we have
\begin{align*}
    (i_2)_!(1) &= -(1-\kappa) x_{1,1} + \zeta_0\cwd - e^{-2}\kappa \zeta_0\cwd x_{1,1} \\
    &= -(1-\kappa)(1-\epsilon) x_{1,1} + \zeta_0\cwd 
\end{align*}
The other calculations are similar, using
\begin{align*}
    \rho((i_1)_!(1)) &= c-y & \rho((i_0)_!(1)) &= y  \\
    (i_1)_!(1)^\GG &= (1,0,0,1) & (i_0)_!(1)^\GG &= (1,0,1,0). \qedhere
\end{align*}
\end{proof}

Finally, the multiplicative structure.

\begin{theorem}\label{thm:22 multiplicative}
As an algebra over $\HH^{RO(\Pi BU(1))}(B U(1))$,  $\HH^\gr(\Qpq 22)$ is generated by
\begin{align*}
    \zeta_{00} &\in \HH^{\Omega_{00}}(\Qpq 22), \\
    \zeta_{11} &\in \HH^{\Omega_{11}}(\Qpq 22), \\
    \zeta_{01} &\in \HH^{\Omega_{01}}(\Qpq 22), \\
    \zeta_{10} &\in \HH^{\Omega_{10}}(\Qpq 22), \text{ and} \\
    x_{1,1} &\in \HH^{2\sigma}(\Qpq 22),
\end{align*}
subject to the facts that
\begin{alignat*}{2}
    &\cwd &\quad&\text{is infinitely divisible by $\zeta_0$,} \\
    &\cxwd &&\text{is infinitely divisible by $\zeta_1$,} \\
    &x_{1,1} &&\text{is infinitely divisible by $\zeta_{00}\zeta_{01}$,} \\
    &x_{1,1} - \zeta_0\cwd &&\text{is infinitely divisible by $\zeta_{00}\zeta_{10}$,} \\
    &x_{1,1} - \zeta_1\cxwd &&\text{is infinitely divisible by $\zeta_{11}\zeta_{01}$,} \\
    &x_{1,1} - e^2 &&\text{is infinitely divisible by $\zeta_{11}\zeta_{10}$,} 
\end{alignat*}
and the following relations:
\begin{align*}
    \zeta_{00}\zeta_{11} &= \zeta_0 \\
    \zeta_{01}\zeta_{10} &= \zeta_1 \\
    x_{1,1}^2 &= e^2 x_{1,1} \\
   \cwd\cxwd &= \tau(\iota^{-2})\zeta_0\cwd x_{1,1}.
\end{align*}
\end{theorem}

\begin{proof}
That the given elements generate follows from Proposition~\ref{prop:22 splitting}.
The divisibility properties of $\cwd$ and $\cxwd$ were shown in \cite{CH:QuadricsII}.
The divisibility of the other elements follow from the following facts, which come from
the definitions and from Lemma~\ref{lem:22 xs}:
\begin{align*}
    x_{1,1} &= (i_3)_!(1) \\
    x_{1,1} - \zeta_0\cwd &= -(1-\kappa)(1-\epsilon)(i_2)_!(1) \\
    x_{1,1} - \zeta_1\cxwd &= -(1-\epsilon)(i_1)_!(1) \\
    x_{1,1} - e^2 &= (1-\kappa)(i_0)_!(1),
\end{align*}
together with, for example, that $i_3^*(\zeta_{00}) = 1 = i_3^*(\zeta_{01})$.
The relations among the $\zeta$s come from the component structure of the fixed set,
while the other two relations are specializations of relations shown in \cite{CH:QuadricsII}.
That these relations suffice to determine the algebra is similar to the argument outlined
for Theorem~\ref{thm:db multiplicative}.
\end{proof}

We can restate this structure in terms of the identification 
\[
    s\colon \Xpq 1{}\times\Xpq 1{} \homeo \Qpq 22
\]
given in \S\ref{sec:notations}.
As in \cite{CH:bt2}, we write $\omega_1$ and $\omega_2$ for the two tautological line bundles over $\Xpq 1{}\times\Xpq 1{}$
coming from the two factors, we write $\clod$ and $\cltd$ for the Euler classes of their duals, and $\cxlod$ and $\cxltd$
for the Euler classes of $\chi\omega_1\dual$ and $\chi\omega_2\dual$, respectively.
We also have the line bundles $\omega_1\tensor\omega_2$ and $\chi\omega_1\tensor\omega_2$, with their respective
Euler classes $\cltensd$ and $\cxltensd$.
Note that we have the following commutative diagram, where $\tensor$ indicates
the map classifying $\omega_1\tensor\omega_2$.
\[
    \xymatrix{
        \Xpq 1{}\times\Xpq 1{} \ar[r]^-s_-\homeo \ar[d] & \Qpq 22 \ar[d] \\
        B T^2 \ar[r]^-\tensor & B U(1)
    }
\]
This allows us to consider $\HH^\gr(\Xpq 1{}\times\Xpq 1{})$ as an algebra over
the ring $\HH^{RO(\Pi BT^2)}(B T^2)$,
with the structure over $\HH^{RO(\Pi BU(1))}(B U(1))$ it inherits from $\Qpq 22$ given by restriction.

Working through the definitions and checking gradings, we can see the following.
\begin{align*}
    \clod &= [s(\Xq{}\times\Xpq 1{})]^* = (\zeta_{00}\zeta_{01})^{-1}[i_3(\Xpq 1{})]^* \\
        &= \zeta_{00}^{-1}\zeta_{01}^{-1}x_{1,1} \\
    \cxlod &= [s(\Xp{}\times\Xpq 1{})]^* = (\zeta_{11}\zeta_{10})^{-1}[i_0(\Xpq 1{})]^* \\
        &= (1-\kappa)\zeta_{11}^{-1}\zeta_{10}^{-1}(x_{1,1} - e^2) \\
    \cltd &= [s(\Xpq 1{}\times\Xq{})]^* = (\zeta_{00}\zeta_{10})^{-1}[i_2(\Xpq 1{})]^* \\
        &= -(1-\kappa)(1-\epsilon)\zeta_{00}^{-1}\zeta_{10}^{-1}(x_{1,1} - \zeta_0\cwd) \\
    \cxltd &= [s(\Xpq 1{}\times\Xp{})]^* = (\zeta_{11}\zeta_{01})^{-1}[i_1(\Xpq 1{})]^* \\
        &= -(1-\epsilon)\zeta_{11}^{-1}\zeta_{01}^{-1}(x_{1,1} - \zeta_1\cxwd) 
\end{align*}
As for the classes $\cltensd$ and $\cxltensd$, define the diagonal and antidiagonal by
\begin{alignat*}{2}
    \Delta&\colon \Xpq 1{}\to \Xpq 1{}\times \Xpq 1{} \qquad& \Delta[a:b] &= ([a:b], [a:b]) \\
    \widetilde\Delta&\colon \Xpq 1{}\to \Xpq 1{}\times\Xpq 1{} &
        \widetilde\Delta[a:b] &= ([a:b],[b:a]).
\end{alignat*}
Then
\begin{align*}
    \cltensd &= [s\widetilde\Delta(\Xpq 1{})]^* = \cwd \\
    \cxltensd &= [s\Delta(\Xpq 1{})]^* = \cxwd
\end{align*}
because we can, for example, represent $\cwd$ by
\[
    \{ [x:u:y:v]\in\Qpq 22 \mid x = u \} = s(\widetilde\Delta(\Xpq 1{})).
\]

These identifications give us corresponding divisibility properties of
$\clod$ and so on. 
We then get the following.

\begin{corollary}\label{cor:x11 times x11}
As an algebra over $\HH^{RO(\Pi BT^2)}(B T^2)$, $\HH^\gr(\Xpq 1{}\times \Xpq 1{})$
is determined by the facts that
\begin{alignat*}{2}
    &\clod &&\text{is infinitely divisible by $\zeta_{00}\zeta_{01}$,} \\
    &\cxlod &&\text{is infinitely divisible by $\zeta_{11}\zeta_{10}$,} \\
    &\cltd &&\text{is infinitely divisible by $\zeta_{00}\zeta_{10}$,} \\
    &\cxltd &&\text{is infinitely divisible by $\zeta_{11}\zeta_{01}$,} \\
    &\cltensd &\quad&\text{is infinitely divisible by $\zeta_{00}\zeta_{11}$,} \\
    &\cxltensd &&\text{is infinitely divisible by $\zeta_{01}\zeta_{10}$,}
\end{alignat*}
and the relations
\begin{align*}
    \clod\cxlod &= 0 \\
    \cltd\cxltd &= 0.
\end{align*}
\end{corollary}

\begin{proof}
Comparing to Theorem~\ref{thm:22 multiplicative}, which gives the structure as a module over $\HH^{RO(\Pi BU(1))}(B U(1))$,
all of the generators listed there already exist in the cohomology
of $B T^2$, though without the divisibility properties, so it suffices to adjoin the implied quotients
(the elements $(\zeta_{00}\zeta_{01})^{-k}\clod$ and so on, with the relations that say how they behave)
to generate the algebra.

The relations $\zeta_{00}\zeta_{11} = \zeta_0$ and $\zeta_{01}\zeta_{10} = \zeta_1$ already hold in
$\HH^{RO(\Pi BT^2)}(B T^2)$.
The other two relations given by the theorem translate into
\begin{align*}
    \clod[2] &= e^2\zeta_{00}^{-1}\zeta_{01}^{-1}\clod \\
    \cltensd\cxltensd &= \tau(\iota^{-2})\zeta_{00}^2\zeta_{01}\zeta_{11}\clod\cltensd.
\end{align*}
The first of these is equivalent to $\clod\cxlod = 0$: If we assume that $\clod[2] = e^2\zeta_{00}^{-1}\zeta_{01}^{01}\clod$, then
\begin{align*}
    \clod\cxlod &= \clod\cdot\zeta_{11}^{-1}\zeta_{10}^{-1}((1-\kappa)\zeta_{00}\zeta_{01}\clod + e^2) \\
    &= \zeta_{11}^{-1}\zeta_{10}^{-1}(-e^2\clod + e^2\clod) \\
    &= 0.
\end{align*}
On the other hand, if we assume that $\clod\cxlod = 0$, then
\[
    \clod[2] = \zeta_{00}^{-1}\zeta_{01}^{-1}\clod((1-\kappa)\zeta_{11}\zeta_{10}\cxlod + e^2) \\
    = e^2\zeta_{00}^{-1}\zeta_{01}^{-1}\clod.
\]

Assuming now that $\clod\cxlod = 0$, we show that $\cltd\cxltd = 0$ is equivalent to
$\cltensd\cxltensd = \tau(\iota^{-2})\zeta_{00}^2\zeta_{01}\zeta_{11}\clod\cltensd$.
Using relations from \cite{CH:bt2}, we can rewrite the latter expression as
\begin{align*}
    \tau(\iota^{-2})\zeta_{00}^2\zeta_{01}&\zeta_{11}\clod\cltensd \\
    &= \tau(\iota^{-2})\zeta_{00}\zeta_{01}\zeta_{11}\clod
        (\zeta_{10}\cxlod + \zeta_{01}\cxltd 
	       - e^{-2}\kappa \zeta_{01}\zeta_{10}\zeta_{11}\cxlod\cxltd) \\
    &= \tau(\iota^{-2})\zeta_{00}\zeta_{01}^2\zeta_{11}\clod\cxltd \\
    &= \tau(\iota^{-2})\zeta_{00}^2\zeta_{01}\zeta_{10}\clod\cltd.
\end{align*}
But we also have from \cite{CH:bt2} that
\begin{align*}
    \cltensd\cxltensd &= \clod\cxlod + \cltd\cxltd + \tau(\iota^{-2})\zeta_{00}^2\zeta_{01}\zeta_{10}\clod\cltd \\
    &= \cltd\cxltd + \tau(\iota^{-2})\zeta_{00}^2\zeta_{01}\zeta_{10}\clod\cltd,
\end{align*}
so $\cltd\cxltd = 0$ if and only if $\cltensd\cxltensd = \tau(\iota^{-2})\zeta_{00}^2\zeta_{01}\zeta_{10}\clod\cltd$.

Thus, the relations given in the corollary are equivalent to those in Theorem~\ref{thm:22 multiplicative},
given the relations known to hold in $\HH^{RO(\Pi BT^2)}(B T^2)$.
\end{proof}

\begin{remark}
Note that $\cltensd\cxltensd \neq 0$, but, as shown in the proof above,
\[
    \cltensd\cxltensd = \tau(\iota^{-2})\zeta_{00}^2\zeta_{01}\zeta_{10}\clod\cltd.
\]
We can check that this is reasonable if we reduce to the nonequivariant case. 
Writing $x_1$ and $x_2$ for the nonequivariant multiplicative generators
of the cohomology of $\PP^1\times\PP^1$, and grading on $\Z$, we have
\begin{align*}
    \rho(\clod) &= x_1 & \rho(\cxlod) &= x_1 \\
    \rho(\cltd) &= x_2 & \rho(\cxltd) &= x_2 \\
    \rho(\cltensd) &= x_1 + x_2 & \rho(\cxltensd) &= x_1 + x_2.
\end{align*}
$x_1^2 = 0$ and $x_2^2 = 0$, but $(x_1+x_2)^2 = 2x_1x_2 \neq 0$.

Identifying $\PP^1\times\PP^1$ with $\Qexp 4$, the usual generators of the cohomology of $\Qexp 4$
would be $y = x_1$ and $c = x_1 + x_2$.
\end{remark}

\begin{remark}
In \cite{CH:QuadricsII} we noted that $1-e^{-2}\kappa x_{1,1}$ is a unit in the cohomology of $\Qpq 22$,
because $(1-e^{-2}\kappa x_{1,1})^2 = 1$.
Identifying $\Qpq 22$ with $\Xpq 1{}\times \Xpq 1{}$, we identified $x_{1,1}$ with
$\zeta_{00}\zeta_{01}\clod$.
If $\pi_1\colon \Xpq 1{}\times \Xpq 1{}\to \Xpq 1{}$ is projection to the first factor, then
\[
    1 - e^{-2}\kappa x_{1,1} =
    1 - e^{-2}\kappa \zeta_{00}\zeta_{01}\clod = \pi_1^*(1 - e^{-2}\kappa\zeta_0 \cwd)
\]
and, as we noted earlier,
$1-e^{-2}\kappa\zeta_0\cwd$ was shown to be a unit in the cohomology of $BU(1)$
in \cite{Co:InfinitePublished}.
\end{remark}

\bibliography{Bibliography}{}
\bibliographystyle{amsplain} 

\end{document}